\documentclass[12pt]{article}
\usepackage{setspace}
\usepackage{etex,mathpazo}
\usepackage{fullpage}
\usepackage{adjustbox}	
\usepackage{arydshln}
\usepackage{pdflscape}
\usepackage{amsmath,amsthm,amssymb} 
\usepackage{natbib} 
\usepackage[font=normalfont,tableposition=top]{caption}
\usepackage{float}
\usepackage{appendix}
\usepackage{color}
\usepackage{tikz}
\usetikzlibrary{arrows,intersections}
\usepackage{fullpage} 
\usepackage{multirow}
\usepackage{graphicx}
\graphicspath{ {Figures/} }  	
\usepackage{graphics}
\usepackage{caption}
\usepackage{subcaption}
\usetikzlibrary{arrows,intersections}

\usepackage{hyperref}
\hypersetup{%
	colorlinks = true,
	linkcolor  = blue,
	citecolor  = blue
}

\makeatletter
\renewcommand*{\eqref}[1]{%
  \hyperref[{#1}]{\textup{\tagform@{\ref*{#1}}}}%
}
\makeatother

\usepackage{chngcntr}
\usepackage{setspace}

\newtheorem{theorem}{Theorem}[section]
\newtheorem{definition}[theorem]{Definition}
\newtheorem{proposition}[theorem]{Proposition}
\newtheorem{corollary}[theorem]{Corollary}
\newtheorem{lemma}[theorem]{Lemma}
\newtheorem{remark}[theorem]{Remark}
\newtheorem{example}[theorem]{Example}
\newtheorem{examples}[theorem]{Examples}
\newtheorem{foo}[theorem]{Remarks}
\newenvironment{Example}{\begin{example}\rm}{\end{example}}

\usepackage{bm}

\usepackage{graphicx}
\usepackage{adjustbox}	

\graphicspath{ {Figures/} }  	
\usepackage{graphics}
\usepackage{multirow}

\usepackage{caption}
\usepackage{subcaption}

\newcommand{\bbC}{\mathbb C}

\newcommand{\bbE}{\mathbb E}
\newcommand{\bbF}{\mathbb F}
\newcommand{\bbG}{\mathbb G}
\newcommand{\bbH}{\mathbb H}

\newcommand{\bbL}{\mathbb L}

\newcommand{\bbP}{\mathbb P}
\newcommand{\bbQ}{\mathbb Q}
\newcommand{\bbR}{\mathbb R}

\newcommand{\scD}{\mathcal D}
\newcommand{\scE}{\mathcal E}
\newcommand{\scF}{\mathcal F}
\newcommand{\scG}{\mathcal G}
\newcommand{\scH}{\mathcal H}

\newcommand{\scL}{\mathcal L}

\newcommand{\scS}{\mathcal S}
\newcommand{\scT}{\mathcal T}

\newcommand{\scV}{\mathcal V}
\newcommand{\scW}{\mathcal W}
\newcommand{\scX}{\mathcal X}

\newcommand{\veps}{\varepsilon}

\newcommand{\norm}[1]{\ensuremath{\left\| #1 \right\|}}

\newcommand{\indicator}[1]{\ensuremath{\mathbf{I}_{#1}}}

\newcommand{\half}{\frac{1}{2}}

\newcommand{\crl}[1]{\ensuremath{ \left\{ #1 \right\} }}
\newcommand{\edg}[1]{\ensuremath{ \left[ #1 \right] }}
\newcommand{\brak}[1]{\ensuremath{\left( #1 \right)}}

\newcommand{\levy}{{L\'evy~}}

  \clearpage 
\usepackage{titling}

\usepackage{booktabs}

\usepackage[left=1in, right=1in, top=1in, bottom=1in]{geometry}
\usepackage{setspace} 
\newcommand\hasan[1]{{\leavevmode\color{black}{#1}}} 
\newcommand\kihun[1]{{\leavevmode\color{black}{#1}}} 
\newcommand\edit[1]{{\leavevmode\color{black}{#1}}} 

\usepackage[T1]{fontenc}
\usepackage{caption}
\allowdisplaybreaks[4]

\usepackage{titlesec}
\titlespacing*{\section}{0pt}{0.3\baselineskip}{0.8\baselineskip}
\titlespacing*{\subsection}{0pt}{0.5\baselineskip}{0.01\baselineskip}

\usepackage{colortbl}

\expandafter\def\expandafter\normalsize\expandafter{%
    \normalsize
    \setlength\abovedisplayskip{3pt}
    \setlength\belowdisplayskip{3pt}
    \setlength\abovedisplayshortskip{50pt}
    \setlength\belowdisplayshortskip{50pt}
}

\usepackage{enumitem}

\usepackage{mathtools}
\mathtoolsset{showonlyrefs}

\usepackage[bottom]{footmisc}

\makeatletter
\def\@makefntext{\hskip 0em\@makefnmark}
\makeatother

\newcommand*{\QEDB}{\hfill\ensuremath{\square}}%

\begin{document}


\begin{titlepage}

\setlength{\droptitle}{-5em}

\title{\textbf{Correlated time-changed \levy processes}\thanks{We would like to thank Julien Hugonnier for motivating this paper and for fruitful discussions. We also thank Samuel Cohen, Yan Dolinsky, Gregoire Loeper, Loriano Mancini, Stoyan Stoyanov, and seminar participants at QMF~2018 and 62nd Annual Meeting of the Australian Mathematical Society for comments. The Centre for Quantitative Finance and Investment Strategies has been supported by BNP Paribas.}}
\date{First draft: August 6, 2018\\
	This draft: \today}
\author{
     Hasan Fallahgoul\thanks{Hasan A. Fallahgoul, Monash University, School of Mathematics and Centre for Quantitative Finance and Investment Strategies, 9 Rainforest Walk, 3800 Victoria, Australia. E-mail: hasan.fallahgoul@monash.edu.}\\{\small Monash University}
     \\
	\bigskip
	\and Kihun Nam\thanks{Kihun Nam, Monash University, School of Mathematics and Centre for Quantitative Finance and Investment Strategies, 9 Rainforest Walk, 3800 Victoria, Australia. E-mail: kihun.nam@monash.edu.}\\{\small Monash University}
	\\
\bigskip}
\noindent
\maketitle

\begin{abstract}
	
	
	
	\cite{Carr:2004hl}, henceforth CW, developed a framework that encompasses almost all of the continuous-time models proposed in the option pricing literature. Their main result hinges on the stopping time property of the time changes, but all of the models CW proposed for the time changes do not satisfy this assumption.
	In this paper, when the time changes are adapted, but not necessarily stopping times, we provide analogous results to CW. We show that our approach can be applied to all models in CW.
\end{abstract}
\thispagestyle{empty}
\end{titlepage}
\setcounter{page}{1}


\setstretch{1.5}
\section{Introduction}

Modeling underlying asset returns is the key component in risk management and derivative pricing. The many papers on this topic reveal the challenges of this task. \cite{black1973pricing} were the first to introduce an elegant option pricing model. They used Brownian motion as the building block for the underlying asset returns. Since then, there have been several departures from the Black-Scholes model. Empirical research established the stylized facts of financial asset returns: asymmetry, jumps that induce fat tails, stochastic volatilities, and the leverage effect.\footnote{The leverage effect refers to the correlation between an asset's return and its changes in volatility. Intuitively, a natural estimate for the leverage effect uses the empirical correlations between the daily returns and the changes in daily volatility. However, \cite{ait2013leverage} show that this natural estimate yields nearly zero correlation, which contradicts many economic reasons for expecting a negative estimated correlation. \cite{fallahgoul2019time} provide a novel class of TCL models that can capture the leverage effect. } The Black-Scholes model fails to capture these stylized facts. Alternatively, \textit{time change \levy (TCL) processes} have been used as a flexible class of models to capture these empirical regularities.\footnote{\cite{Huang:in} investigate the specifications of different option pricing models under TCL processes. \cite{bates2012us} compares four different TCL processes and captures the stochastic volatility and left tail movements in US stock market returns. \cite{ornthanalai2014levy} estimates discrete-time models in which asset return innovations follow both Brownian motion and TCL processes. \cite{fallahgoul2019time} provide an extensive time series and option pricing analysis of a novel class of TCL models.} 

This study investigates the use of TCL when the innovations of the \levy process are correlated with innovations of its stochastic time change process (correlated TCL). \kihun{Neither the transitional density nor the characteristic function for a correlated TCL process is available in a closed-form, and thus, it is challenging to use a correlated TCL in a financial model.} To the best of our knowledge, \cite{Carr:2004hl}, henceforth CW, is the only study that tackles this problem. Their main result hinges on the stopping time property\footnote{For a given filtration $(\scF_s)_{s\geq 0}$, we say $ T_t $ is a stopping time if $\crl{T_t\leq s}\in\scF_s$ holds for all $s$.} of the time changes, but all of the models that CW propose for the time changes do not satisfy this assumption.\footnote{See \cite{fn2019} for detailed discussion.} \kihun{In this study, we first show that the specification analysis in CW is not compatible with their assumptions. Then, we propose correlated TCL processes as a model for asset returns in which the time changes are only an adapted process. We next provide the sufficient conditions to guarantee no arbitrage and discuss the hedging of contingent claims.} Finally, for a given solution to a linear parabolic partial differential equation (PDE), we derive a closed-form expression for the transitional density, moment generating function, and characteristic function of the correlated TCL processes. Our approach can apply to virtually all models in the option pricing literature.

Let us briefly review \cite{fn2019} regarding the problem of CW. \hasan{Even though the characteristic function of CW seems to be correct when the time changes are stopping times, none of the specifications they proposed for the time changes satisfies this condition. In addition, it is not clear how to construct the stochastic time $ T=(T_t)_{t\geq 0}$ satisfies all the assumptions of CW while the Laplace transforms of $T_t$s with respect to the \textit{leverage-neutral complex measure} can be expressed in closed form.\footnote{The leverage-neutral measure is $\bbQ(\theta)$ defined by $d\bbQ(\theta)=\exp(i\theta^TY_t+T_t\Psi_x(\theta))d\bbP$ where $Y=X_T$ is the \levy process runs on stochastic time $T$ and $\Psi_x(\theta)$ is the characteristic exponent of $X_t$} Moreover, such specified $T$ should capture the empirical regularities we mentioned above. 
	
It is tempting to consider an enlarged filtration to make the time changes stopping times. However, when the \levy process $X=(X_t)_{t\geq0}$ and the process of the time changes $T=(T_t)_{t\geq0}$ are dependent, enlarging the filtration may affect the semimartingale property of the process $X$ through the dependency.\footnote{For example, let $X$ be a Brownian motion under its natural filtration and $T_t:=\int_0^te^{X_s-s^2/2}ds$. The smallest filtration that makes $T_t$ stopping times is $\bbG:=(\scF^X_{C_t})_{t\geq 0}$, where $(\scF^X_t)$ is the filtration generated by $X$ and $C_t$ is the first $t$-level crossing time of $T$ (see Proposition \ref{pre_stopping_adapted}). This implies that given $\bbG$, one can determine the path of $X$ from time $0$ to $C_t$ at time $t$. If $C_t>t$ happens at time $t$, then we know the path of $X$ from $t$ to $C_t$ at time $t$. Since $X$ has infinite first-order variation, $X$ cannot be a semimartingale with respect to $\bbG$.} If $X$ is not a semimartingale, then one has an arbitrage opportunity by Theorem 7.2 of \cite{delbaen1994general}. \footnote{For example, deterministic continuous path with infinite first order variation cannot be a semimartingale.}}

The stopping time assumption of $T_t$ for all $t$ is extremely restrictive in financial applications. \kihun{The problem is that, the business activity rate $v_t$ at time $t$ (the time derivative of $T$ at time $t$) is known at time $t$. This implies that $T$ is adapted to the underlying filtration.} If $ T $ is a collection of stopping times, then we should be able to foretell, not forecast, the future of the business activity rate as well as that in the past (see Proposition \ref{pre_stopping_adapted}). This violates commonsense in finance and through the correlation with $X$, it may generate arbitrage.

Let us clarify with an example. Let the business activity rate $ v $ be a Cox-Ingersoll-Ross (CIR) process, which is an affine activity rate model in Section 4.2.1 of CW. The left panel of Figure \ref{simu_cir} shows a simulation of the CIR process using parameters taken from Table 1 in \cite{fallahgoul2019time}. The right panel shows the integral of the simulated business activity rate that represents $T$. Since the calendar time $ t $ is replaced by the time change $ T_t $, one needs to impose the condition that the unconditional expectation of $ T_t $ is $ t $. Therefore, a realized path of $T$ oscillates around the line $y=t$, as the right panel of Figure \ref{simu_cir} shows.
Since the business activity rate $v$ is adapted, its integrated process $T$ is adapted as well. If $ T_t $ is a stopping time for each $t$, then based on Proposition \ref{pre_stopping_adapted}, we know the path of $ T $ until it hits the level $ t $ at time $t$. In the right panel of Figure \ref{simu_cir}, $ T_t $ is greater than $t$ at time $ t=1.2 $. The adaptedness of $T$ implies that we know the whole path until $ S $. Since the stopping time property of $T_t$ for each $t$ requires knowledge of the path only until $R$, then there is no problem. On the other hand, $ T_t $ is less than the $t$ line at time $ t=0.6$. Then, $ T_t $ being a stopping time for each $t$ implies that we know the path of $T$ until $Q$ instead of $P$. This means that, at time $t=0.6$, we can foresee the future path of $T$ until it reaches $Q$. This is impossible under the filtration on which the CIR process is defined. 

\begin{center}
	\fbox{Insert Figure~\ref{simu_cir} about here}
\end{center}


In this paper, we study the correlated TCL framework when $ T $ is adapted, but $ T_t $ are not stopping times. This relaxed assumption enables us to adopt most dynamic models for time changes, including all major examples in CW. However, this extension is not obvious for two reasons:
\begin{itemize}
	\item The log-return $Y=X_T$ may no longer be a semimartingale. This imposes two problems. First, the wealth process of a portfolio represented as a stochastic integral with respect to the asset price processes, is not well defined. Moreover, the no arbitrage property \kihun{of predictable trading strategies} implies $Y$ is a semimartingale (see Theorem 7.2 of \cite{delbaen1994general}). \kihun{Therefore, we may have arbitrage if our trading strategies can be any predictable process.}
	\item Since the time changes are not stopping times, we cannot use the optional stopping theorem. Consequently, one cannot expect a simple formula, as in CW, and we need a new approach to calculate the distribution of the correlated TCL processes.	
\end{itemize}
We use the following approaches to overcome these difficulties:
\begin{itemize}
	\item We restrict ourselves to \kihun{a class of} simple predictable investment strategies as an integrand so we can define stochastic integrals with respect to non-semimartingales. In particular, by disallowing continuous trading, we can prove that there is no arbitrage. \kihun{Note that 
		continuous trading is impossible in the real-world due to the non-zero transaction cost and information processing time. In addition, we show that any predictable hedging strategy can be approximated with our class of admissible trading strategies (Cheridito Class).}
	\item We use the Fokker-Planck PDEs to find the probability density functions (PDFs), the Laplace transforms, and the characteristic functions of the TCL processes in terms of the corresponding PDE solutions. In particular, by finding an explicit measure change, we decouple the correlation between the {\levy} process and the time change, which is in parallel with CW's main result. 	
\end{itemize}

In summary, our contributions in this paper are two-fold. First, we discuss why the framework in CW is extremely restrictive and why their proposed specifications cannot be used in their framework. Second, we provide an alternative framework in which we model the time changes by adapted processes without being stopping times. We find the distribution of the TCL processes, provided no arbitrage conditions, and a specification analysis.

\edit{The structure of the article is as follows. In Section 2, we discuss the relationship between adaptedness and the stopping time property for time changes. Then, we establish the basic settings of our model in Section 3. \kihun{Section 4 is devoted to the no arbitrage property for the correlated TCL process and the discussion about hedging in our framework. In Section 5, we provide formulas for the PDF and the Laplace transform (and the characteristic function) of a correlated TCL process.} In Section 6, we generalize the framework of \cite{Huang:in} under our setting. Finally, we conclude with a summary and future research questions in Section 7. 
}

\section{Adaptedness and the stopping time property}

In this section, we show that the specifications for the stochastic time $T$ in CW does not guarantee that $T_t$ is a stopping time for each $t$.

Let $\bbF=(\scF_t)_{t\geq 0}$ be the underlying filtration in CW.\footnote{Throughout CW, it is not clear which filtration they refer to when they mention martingale, stopping time, and \levy process. In order to make Lemma 1 of CW correct, we need to assume $T_t$ are stopping times with respect to the filtration $\bbF$ such that $X$ is a \levy process, while CW assume $T$ is adapted to $\bbF$ in Section 4.3.1. For consistency, we assume that CW used the same filtration throughout their paper.} \kihun{In their paper, $T$ satisfies two assumptions: $T_t$ is a $\bbF$-stopping time for each $t$, and for $s\leq t$, $T_s$ is $\scF_t$-measurable.}
They use the first assumption to derive the characteristic function or generalized Fourier transform of the TCL process (see Lemma 1 of CW), and use the second assumption for the specifications of the business activity rate $v_t=\frac{dT_t}{dt}$ in Section 4.2 of CW.

Unfortunately, these two assumptions imply that at time $t$, we can determine when the stochastic time $T$ hits the $t$-level, and moreover, the path of $T$ until it hits the $t$-level (see Proposition \ref{pre_stopping_adapted}). This property is rarely satisfied in a stochastic differential model of $T$ unless it is deterministic.

\begin{proposition}\label{pre_stopping_adapted}
	Let $(T_t)_{t\geq 0}$ be an increasing continuous process. Assume that $\bbF:=(\scF_t)_{t\geq 0}$ is a filtration such that $T$ is adapted and $T_t$ is a $\bbF$-stopping time for each $t\geq 0$. Let $C_s:=\inf\crl{t> 0: T_t>s}$ be the first $s$-level crossing time of $T$. Then, for all $t\geq 0$, $C_t$ is $\scF^T_t$-measurable. Moreover,
	\begin{align*}
	\scF^T_{C_t}\subset \scF^T_t\subset\scF_t
	\end{align*}
	where $(\scF^T_t)_{t\geq 0}$ is the filtration generated by $T$.
\end{proposition}
\begin{proof}
	Note that $C_t$ is $\scF^T_t$-measurable for every $t\geq 0$ because
	\begin{align*}
	C_t^{-1}([s,\infty))=\crl{C_t\geq s}=\crl{T_s\leq t}\in\scF^T_t
	\end{align*}
	for any $s\in\bbR$, and $\crl{[s,\infty):s\in\bbR}$ generates Borel sigma algebra in $\bbR$. Therefore, $\scF^C_t\subset\scF^T_t$, where $(\scF^C_t)_{t\geq 0}$ is the filtration generated by $C$. We now prove $\scF^T_{C_t}\subset \scF^C_t$. 
	
	Since $T$ is a continuous increasing function,
	$
	T_s=\inf\crl{u\geq 0:C_u>s}.
	$
	Therefore, $T_s$ is the first $s$-level crossing time of $C$. This implies that
	$
	\scF^T_s\subset\scF^C_{T_s}.
	$
	Now, if we let $s=C_t$ and use the relationship $T_{C_t}=t$, we have
	\begin{align*}
	\scF^T_{C_t}\subset \scF^C_t,
	\end{align*}
	which proves the claim.
\end{proof}
The above proposition implies that if $T$ is an adapted collection of stopping times, then the first $t$-level crossing time of $T$ can be determined at time $t$. None of the specifications in CW seems to satisfy this extremely restrictive condition. For example, we cannot use autonomous stochastic differential equation for the specification of the business activity rate $v$.
\begin{Example}
	Let $B$ be a Brownian motion \kihun{and $\bbF=(\scF_t)_{t\geq 0}$ be the corresponding Brownian filtration. Furthermore, }
	assume that $v_0=1$ and
	\begin{align*}
	dv_t = \mu(v_t)dt+\sigma (v_t)dB_t
	\end{align*}
	for some measurable functions $\mu,\sigma:\bbR\to\bbR$, which guarantee the existence of a unique strong solution $v$ that is nonnegative, $\sigma(v_t)$ is nonzero for every $t\geq0$, and $T_t:=\int_0^tv_sds$ is an increasing process with $\bbE T_t=t$. Then, $\brak{T_t}_{t\geq 0}$ cannot be a collection of $\bbF$-stopping times.
\end{Example}
\begin{proof}
	For a stochastic process $Z$, let us denote $\bbF^Z=(\scF^Z_t)_{t\geq 0}$ as the filtration generated by $Z$.
	First, note that
	\begin{align*}
	dB_t=-\frac{\mu(v_t)}{\sigma(v_t)}dt+\frac{1}{\sigma(v_t)}dv_t
	\end{align*}
	for a given solution $v$. Therefore, $B$ is adapted to $\bbF^v$, implying that $\scF_t=\scF^v_t=\scF^T_t$.
	Assume that $T_t$ is a $\bbF$-stopping time for each $t\geq 0$. Then, by the previous proposition, $C_t:=\inf\crl{u>0:T_u>t}$ is $\scF_t$-measurable. On the other hand, recall that $C_t$ is the first $t$-level crossing time of $T$ and it can be bigger than $t$. If $C_t$ is bigger than $t$, then $C_t$ depends on the path of $B$ after $t$ as well. This contradicts that $C_t$ is $\scF_t$-measurable.
\end{proof}
\begin{Example}
	Consider a filtered probability space $(\Omega,\scF,\bbF,\bbP)$, where $\bbF$ is generated by a Brownian motion $B$. In addition, assume that $T_t=\int_0^tv_sds$, where $v_0=1$ and
	\begin{align*}
	dv_s = \kappa(1-v_s)ds+\sigma\sqrt{v_s}dB_s.
	\end{align*}
	Here, we assume the Feller condition, $ 2\kappa>\sigma^2$.
	Then, $(T_t)_{t\geq 0}$ cannot be a collection of $\bbF$-stopping times.
\end{Example}

\kihun{In examples above,} if we enlarge the filtration $\bbF$ so that $\scF_t$ contains information about the future path of $T$ until it hits level $t$, then $T_t$ become stopping times. However, if this is the case, then $B$ is no longer a Brownian motion under this enlarged filtration since the dependency between $T$ and $B$ will affect the distribution of the future path of $B$. \kihun{Likewise, one cannot enlarge the filtration for the specifications in CW, as the enlargement can make $X$ a non-{\levy} process.}

\section{Correlated TCL processes}\label{setting}

We write a \levy process in the form of 
\begin{align}\label{levy_process}
X_t = \alpha J_t+ \beta Z_{J_t}
\end{align}
for a non-decreasing right-continuous process with left-limits given by $ J$ and a Brownian motion $ Z$ that is independent of $J$. The process $ J $ is called a subordinator.\footnote{Detailed information about a subordinator process can be found in \cite{cont2004financial}.} Many well-known \levy processes can be obtained from equation \eqref{levy_process}. For example, we can recover the variance-gamma (\cite{madan1990variance}), normal-inverse Gaussian (\cite{barndorff1997processes}), and CGMY (\cite{carr2003stochastic}) process by assuming that $ J_t $ is a Gamma, inverse Gaussian, and tempered stable process, respectively.\footnote{The related subordinator for the other processes in Table 1 of \cite{Carr:2004hl} can be obtained using the same technique as in \cite{madan2006cgmy}.}

The \levy process $ X$ is not the TCL process $Y$ defined in \cite{Carr:2004hl}, although it is obtained by a time change. The time change in a \eqref{levy_process} captures the jump regularity of the log returns, while the time change in \cite{Carr:2004hl} captures the mean-reversion and stochastic volatility of the log returns.



\begin{definition} 
	A correlated TCL process $Y$ for a \levy process $X$ in \eqref{levy_process} is given by $Y_t:=X_{T_t}$, where $ T $ is a non-decreasing right-continuous process with left-limits.
\end{definition}

In general, the stochastic time (or random time change) $ T $ can be modeled as a non-decreasing process. For simplicity, let us suppress jumps and assume absolute continuity in time.
In other words,
\begin{align}
T_t=\alpha_t+\int_{0}^{t}v_{s-}ds
\end{align}
where $v$, a right-continuous process with left-limits, is the instantaneous activity rate. The random variables $ T_t $ and $ t $ are the business time and calendar time, respectively.


\kihun{As we described in the Introduction, we study the case in which $T$ is just an adapted process and $T_t$ are no longer stopping times.}
To be more specific, we assume the following conditions in this article. Let $(\Omega,\scF,\bbF,\bbP)$ be a filtered probability space where $\bbF=(\scF_t)_{t\geq 0}$ satisfies the usual conditions, two independent Brownian motions $B$ and $W$, and an adapted non-decreasing \levy process $J$ that is independent of these Brownian motions. For $\rho\in[-1,1]$, let $Z_t:=\rho B_t+\sqrt{1-\rho^2}W_t$.
For $\mu,\sigma:[0,\infty)\times\bbR\to\bbR$, consider the processes $v_t$ and $T_t$ given by the following SDE:
\begin{align}\label{sde}
dv_t&=\mu(t,v_t)dt+\sigma(t,v_t)dB_t,& v_0=1,\\
dT_t&=v_tdt,& T_0=0.
\end{align}
Here, we assume appropriate conditions for $\mu$ and $\sigma$ to guarantee the existence of a unique strong solution such that $v$ is nonnegative, $\bbE T_t=t$, and $T_t\xrightarrow{t\to\infty}\infty$ a.s. Since $(v,T)$ is a Markov process, we define $\Phi:[0,\infty)\times\bbR\to\bbC$ as the function satisfying
\begin{align*}
\bbE \edg{e^{i\xi T_t}\vert\scF_s}=\exp(\Phi(t,s,\xi,v_s,T_s)), \qquad s\leq t.
\end{align*}
\begin{remark}\label{nonadapted}
	Under our assumptions, the $X$ defined in \eqref{levy_process} is not adapted to the filtration $\bbF$ in general. The adaptedness of $X$ to $\bbF$ is not necessary: even if $X$ is adapted to $\bbF$, its time-changed process $Y_t=X_{T_t}$ will not be adapted to $\bbF$ since $T_t>t$ with positive probability when $T$ is random. If we need to the adaptedness of $Y$, then one can consider the filtration generated by $Y$ after its construction. In Section 4, when we discuss the arbitrage property, we will consider $\bbF^{Y,T}$; that is, the filtration generated by $Y$ and $T$.
\end{remark}
\begin{remark}\label{propindep}
	The case in which $J$ is an increasing \levy process is of special interest since it is the simplest model that captures extreme events (producing the heavy tail property). Note that in this case, $J$ is the sum of a deterministic drift and a pure jump {\levy} process by the L\'evy-It\^o decomposition theorem. In addition, Theorem 11.43 of \cite{sheng:1992semi} tells us that if $X$ and $\tilde{X}$ are adapted \levy processes with $[X,\tilde{X}]=0$, then $X$ and $\tilde{X}$ are independent. This result implies that if $J$ is an increasing \levy process, then $J$ is automatically independent with respect to $(B,W,Z)$.
	Note that if $\rho\neq 0$, then the TCL process $Z_J$ is not independent of $B$, even though $Z_J$ is a pure jump \levy process, while $B$ is a continuous \levy process. Therefore, when $\rho\neq 0$, there is no filtration that makes $Z_J$ and $B$ adapted simultaneously. Note that as in Remark \ref{nonadapted}, we focus on the filtration such that $B$ is adapted, while our $X$ is a non-adapted process with respect to $\bbF$, but a \levy process under the filtration generated by $X$.

\end{remark}

We end this section with our notation. For any random variable $\xi$, we use notation $f_\xi$ as a probability density ``function" (PDF); that is,
\begin{align*}
\bbP(\xi\in A):=\int_A f_\xi(x)dx.
\end{align*}
We also use $f_{X|Y}$ as the conditional PDF of $X$ with respect to $Y$.
We also use notation $\int$ for an integral on the whole domain unless otherwise stated. 

\section{Arbitrage and hedging}

\subsection{No arbitrage property}\label{noarbitrage}
{\hasan{Note that our model of asset prices may not be a semimartingale.} For a frictionless market in which continuous trading is possible, the asset price should be a semimartingale in order to avoid arbitrage opportunities (see Theorem 7.2 of \cite{delbaen1994general}). On the other hand, in a realistic model, continuous trading is impossible and the market has friction due to transaction costs. Therefore, the no arbitrage condition may be satisfied for a more general class of asset price processes. \cite{cheridito2003arbitrage} originally used the idea to prove no arbitrage when the asset price follows a fractional Brownian motion, which is not a semimartingale.
	
	In this section, using the results from \cite{jarrow2009no}, hereafter JPS, we will provide a sufficient condition for no arbitrage in our framework in which the time change process $T_t$ is no longer the stopping time. We state it rigorously by introducing the following notations. First, for a given process $X$, we denote $\bbF^X=(\scF^X_t)_{t\geq 0}$ as the augmented filtration generated by $X$. 
	\begin{definition}
		For a given filtration $\bbF$ and any $h>0$, let 
		\begin{align*}
		\scS(\bbF):=&\Big\{g_0{\bf 1}_{0}+\sum_{j=1}^{n-1}g_j{\bf 1}_{(\tau_j,\tau_{j+1}]}:n\geq 2,\\
		& 0\leq\tau_1\leq\cdots\leq\tau_n \text{where all $\tau_j$ are $\bbF$-stopping times,}\\
		&\text{$g_0$ is a real number and $g_j$ are $\scF_{\tau_j}$-measurable random variables, and}\\ &\text{$\tau_n$ is bounded}\Big\}\\
		\scS^h(\bbF):=&\Big\{g_0{\bf 1}_{0}+\sum_{j=1}^{n-1}g_j{\bf 1}_{(\tau_j,\tau_{j+1}]}\in\scS(\bbF):\min_{j}(\tau_{j+1}-\tau_j)\geq h \text{ a.s.}\Big\}.\\
		\Pi(\bbF):=&\cup_{h>0} S^h(\bbF).
		\end{align*}
		We call $\Pi(\bbF)$ the Cheridito Class of trading strategies.
	\end{definition}
	Set $\scS(\bbF)$ is the set of trading strategies with a finite number of transactions and $\scS^h(\bbF)$ is the set of trading strategies in which the timing of two trades are at least $h$ apart. $\Pi(\bbF)$ is the set of trading strategies in which there is a minimum time interval between trades. Therefore, continuous trading is disallowed if we restrict ourselves to $\scS(\bbF), \scS^h(\bbF)$, or $\Pi(\bbF)$.
	\begin{definition}
		For $\bbF$-adapted process $S$, we say $S$ satisfies the no arbitrage property with respect to $\scX$($=\scS(\bbF)$ or $\Pi(\bbF)$) if
		\begin{align*}
		\crl{(H\cdot S)_\infty: H\in \scX}\cap L_0^+=\crl{0},
		\end{align*}
		where $H=g_0{\bf 1}_{0}+\sum_{j=1}^{n-1}g_j$, $H\cdot S = g_0S_0+\sum_{j=1}^{n-1}g_j(S_{\tau_{j+1}}-S_{\tau_j})$, and $L_0^+$ is the set of $\scF$-measurable non-negative random variables.
	\end{definition}
	\begin{remark}
		Note that we do not require $S$ to be a semimartingale, but only adapted. This is possible because we focus only on simple predictable trading strategies.
	\end{remark}
	If we assume the no arbitrage property of $X$ with respect to $\scS(\tilde\bbG)$, where $\tilde\bbG$ is a time-changed filtration of $\bbF^{Y,T}$, we can prove the no arbitrage property under $\scS(\bbF^{Y,T})$ using Theorem 6 in JPS.
	Since $T$ is $\bbF^{Y,T}$-adapted, the first $t$-level crossing time $C_t:=\inf\crl{u>0: T_u>t}$ is a $\bbF^{Y,T}$-stopping time for each $t\geq 0$. Then, we can define $\tilde\bbG:=\brak{\tilde\scG_t}_{t\geq 0}$, where
	$\tilde\scG_t:=\scF^{Y,T}_{C_t}$. Since $T$ is a continuous strictly increasing function, we have $T_{C_t}=t, C_{T_t}=t,$ and $Y_{C_t}=X_{T_{C_t}}=X_t$. Therefore, $X$ is adapted to $\tilde\bbG$.
	\begin{theorem}\label{nflvr2}
		Assume that $X_u:=\alpha J_u+\beta Z_{J_u}$ has the no arbitrage property with respect to $\scS(\tilde\bbG)$. Then, $e^{Y_t}$ has the no arbitrage property with respect to $\scS(\bbF^{Y,T})$.
	\end{theorem}
	\begin{proof}
		Note that $C$ is adapted to $\tilde\bbG$. Therefore,
		\begin{align*}
		\crl{T_t\leq s}=\crl{t\leq C_s}\in\tilde\scG_s
		\end{align*}
		for all $s$. This implies that $T_t$ is a $\tilde\bbG$-stopping time for each $t$. Then, applying Corollary 5 and Theorem 6 in JPS proves the claim.
	\end{proof}
	Unfortunately, it is not easy to show whether $X$ has the no arbitrage property with respect to $\scS(\tilde\bbG)$ unless $T$ and $Z$ are independent.\footnote{\cite{fallahgoul2019time} derive a necessary and sufficient condition to drive a risk-neutral measure when $ X $ and $ T $ are independent. } Note that $\tilde\scG_t$ contains information about the path of $T$ until it hits the level $t$. This implies that under the information inflow $\tilde\bbG$, we may know something about the path of $X$ until $\max(t, C_t)$ through the dependency between $T$ and $Z$.
	\begin{Example}
		Let $X=B$ and $T_t=\int_0^t \exp\brak{B_s-\half s^2}ds$ for a Brownian motion $B$. Then $X$ has arbitrage with respect to $\scS(\tilde \bbG)$.
	\end{Example}
	\begin{proof}
		Note that since $X=B$ is adapted to $\tilde\bbG$, $T$ is adapted to $\tilde\bbG$.
		Let $u$ be the smallest number such that $\int_0^u\exp\brak{-\half s^2}ds\leq \frac{u-1}{e}$. We also define a $\tilde\scG_u$-measurable random variable $C:=\min\brak{\inf\crl{t>u:T_t>u}, M}$, where $M$ is a number far bigger than $u$. Here, $M$ represents the time horizon of our trading strategies.
		
		Note that the path of $T$ from time 0 to $C$ is $\tilde\scG_u$-measurable. Since
		$X_s=B_s=\log(\partial_t T_t|_{t=s})+\half s^2$,  $X_C$ is also $\tilde\scG_u$-measurable.
		We denote $E$ as an event that $B$ stays below 1 until time $u$ and $X_C>X_u$. Note that $E\in\tilde\scG_u$ and $\bbP(E)>0$. At time $u$, if $E$ did not occur, then do nothing. If $E$ occurs, then buy 1 asset at time $u$ and sell it at time $C$. Then, we gain $X_C-X_u$, which is an arbitrage.
	\end{proof}
	If we further restrict ourselves to the trading strategy $\Pi(\bbF^{Y,T})\subsetneq\scS(\bbF^{Y,T})$, then can provide sufficient conditions for the no arbitrage property when $T$ and $Z$ are dependent.
	\begin{lemma}\label{lmn1}
		Assume the setting in Section \ref{setting} with $\beta\neq 0$ and $\rho\neq 1$. Let $\bar X_t:=\alpha J_t+\beta\sqrt{1-\rho^2}W_{J_t}=\bar X^c_t+\bar X^j_t$, where $\bar X^c$ is the continuous part of $\bar X$ and $\bar X^j$ is the jump part of $\bar X$. Let $\bbG:=(\scG_t)_{t\geq 0}$, where
		$\scG_t:=\scF^{\bar X^c}_t\vee \scF^{B, \bar X^j,J}_\infty.$ 
		If $J$ is an increasing {\levy}process with a nonzero continuous part, then $X=\bar X+\beta\rho B_J$ has the no arbitrage property with respect to $\Pi(\bbG)$.
	\end{lemma}
	\begin{proof}
		Since $J$ is an increasing {\levy}process, $\bar X$ is a {\levy}process. Therefore, $\bar X^c$ is a scaled Brownian motion with a constant drift, which is independent of $\bar X^j$. Then, $\bar X^c$ has the no arbitrage property with respect to general $\bbF^{\bar X^c}$-predictable integrands. Moreover, for all $h\geq 0$, $[\bar X^c,\bar X^c]_{t+h}-[\bar X^c,\bar X^c]_t\geq ch$ for some constant $c$ and $[\bar X^c,\bar X^c]_t$ is bounded for each $t$.
		
		Moreover, $J$ is a sum of a deterministic function and a pure jump {\levy}process. Therefore, $J$ is independent of $B$ and $\bar X^c$ by Theorem 11.43 of \cite{sheng:1992semi}.
		Since $\bar X^j_t, B_t, J_t\in\scG_0$ for any $t$, $V:=\bar X^j+\beta\rho B_{J}$ is a $\bbG$-adapted process and its path is $\scG_0$-measurable. This implies that the path of $V$ is known at time $0$ under the filtration $\bbG$. Therefore, at $t=0$, we can determine the bound of $V$ on any finite time interval. On the other hand, since $\bar X^j,B,J$ are independent of $\bar X^c$, $\bar X^c$ is a $\bbG$-semimartingale that has the no arbitrage property with respect to general $\bbG$-predictable integrands. Therefore, by Theorem 1 of JPS, $X$ has the no arbitrage property with respect to $\Pi(\bbG)$.
	\end{proof}
	\begin{theorem}\label{nflvr} Assume the conditions in Lemma \ref{lmn1}. Assume that $T_{t_0+h}-T_{t_0}\geq \delta(h)>0$ for any $t_0, h>0$, where $\delta(h)$ is an increasing function with $\delta(0)=0$ and $\delta(h)>0$.
		$S_t:=S_0e^{\theta Y_t}$ has the no arbitrage property with respect to $\Pi(\bbF^{Y,T})$.
	\end{theorem}
	\begin{proof}
		We use the notation in Lemma \ref{lmn1}. Note that $\scG_0$ contains all of the information of the path of $B$ from time zero to the end of time. Therefore, $\scG_0$ has all of the information of $T$ and thus, $T_t$ is a $\bbG$-stopping time and $\scH_t:=\scG_{T_t}$ is well-defined. On the other hand, since $Y_t=X_{T_t}$ and $T_t$ are $\scG_{T_t}$-measurable, we have $\bbF^{Y,T}\subset\bbH:=(\scH_t)_{t\geq 0}$. Therefore, by Theorems 2 and 3 of JPS, we need only prove the no arbitrage property of $Y$ with respect to $\Pi(\bbH)$. 
		
		%
		Let $C_s:=\inf\crl{u>0: T_u>s}$, which is a $\bbH$-stopping time for each $s\geq 0$. 
		For any bounded $\bbH$-stopping time $\tau$, we have $\crl{T_\tau\leq s}=\crl{\tau\leq C_s}\in \scH_{C_s}= \scG_s$. Therefore, $T_\tau$ is a $\bbG$-stopping time. Let $\tau_1$ and $\tau_2$ be two bounded $\bbH$-stopping times with $\tau_1\geq \tau_0+h$ for $h>0$. Then, $T_{\tau_1}-T_{\tau_0}\geq \delta(h)$.
		Since $X$ has the no arbitrage property with respect to $\Pi(\bbG)$, by applying Lemma 1 of JPS to $\bbG$-stopping times $T_{\tau_1}$ and $T_{\tau_0}$, for any $A\in\scH_{\tau_0}=\scG_{T_{\tau_0}}$, $\eta:={\bf 1}_A\brak{Y_{\tau_1}- Y_{\tau_0}}={\bf 1}_A\brak{X_{T_{\tau_1}}- X_{T_{\tau_0}}}$ satisfies neither of the following conditions:
		\begin{itemize}
			\item $\bbP(\eta \geq 0 ) = 1$ and $\bbP ( \eta>0 ) >0$, nor
			\item $\bbP ( \eta\leq 0 ) = 1$ and $\bbP ( \eta<0 ) > 0$.
		\end{itemize}
		Again, by Lemma 1 of JPS, we obtain the no arbitrage property of $Y$ with respect to $\Pi(\bbH)$ and applying Theorem 2 of the same paper proves the claim.
	\end{proof}
	
	We can extend the results above when the time changes for the continuous part of $X$ and the jump part of $X$ are different:
	We let
	\begin{align*}
	Y_t:= X^c_{T^c_t}+X^j_{T^j_t},
	\end{align*}
	where
	\begin{align}
	X^c_t&=a^c_1t+a^c_2B^c_t+a^c_3B^j_t+a^c_4 W^c_t,\\
	dv^c_t&=\mu^c(1-\veps-v^c_t)dt+\sigma^c\sqrt{v^c_t(1-\rho^2)}dB^c_t+\rho\sigma^c\sqrt{v^c_t}dB^j_t\\
	dv^j_t&=\mu^j(1-v^j_t)dt+\sigma^j\sqrt{v^j_t}dB^j_t\\
	T^c_t&=\int_0^t(v^c_s+\veps)ds \text{, and}\\
	T^j_t&=\int_0^tv^j_sds
	\end{align}
	for given constants $\rho, \veps$, and $(a^x_i,\mu^x,\sigma^x)_{x\in\crl{c,j}, i=1,2,...}$ with $a^c_4\neq 0, \veps>0$, and independent Brownian motions $B^c,B^j,W^c, W^j$. Note that the continuous part of the \levy process is always the sum of drift and Brownian motion due to the L\'evy-It\^o decomposition theorem. Additionally, Remark \ref{propindep} guarantees the independence of $X^j$ with any Brownian motion in this setup, as well as $X^c,T^c$ and $T^j$.
	The positive $\veps$ guarantees that $T^c$ increases with a rate bounded below so that $Y$ has a continuous part. Note that the claim above holds for any arbitrarily small positive $\veps$. If we let $\veps$ go to $0$, then the distribution of $Y$ will be approximated to the form in Section \ref{HW}.
	\begin{theorem}\label{thm:noarbittcj}
		For any $\veps\in(0,1)$, $e^Y$ has the no arbitrage property with respect to $\Pi(\bbF^{Y, T^c, T^j})$.
	\end{theorem}
	\begin{proof}
		Note that we need only prove the no arbitrage property of $Y$.
		Let 
		\begin{align*}
		\tilde X_t:= X^c_t+X^j_{T^j_{U^c_t}}=S_t+V_t,
		\end{align*}
		where $U^c_t:=\inf\crl{s>0: T^c_s>t}$,$
		S_t:=a^c_1t+a^c_4 W^c_t$, and
		$V_t:=a^c_2B^c_t+a^c_3B^j_t+X^j_{T^j_{U^c_t}}$. Note that $U^c_{T^c_t}=t$ since $T^c$ is a continuous increasing process.
		
		Define $\bbG:=(\scG_t)_{t\geq 0}$ with
		\begin{align*}
		\scG_t:=\scF^{W^c}_t\vee \scF^{X^j, B^j,B^c}_\infty.
		\end{align*}
		Then, $V_t$ is $\scG_0$-measurable for all $t\geq 0$, and therefore, $\tilde X$ is $\bbG$-measurable. In particular, the bound of $V$ on any finite time interval is $\scG_0$-measurable and $V$ is independent of $S$.
		
		Since $[S,S]_t=a^c_4t$ is bounded on any finite time interval and $[S,S]_{t+h}-[S,S]_t=a^c_4 h$, we can apply Theorem 1 of JPS to conclude that $\tilde X$ does not have arbitrage with respect to $\Pi(\bbF^{Y, T^c, T^j})$.
		
		Note that $T^c_{t_0+h}-T^c_{t_0}\geq \veps h$. Moreover, $Y_t=\tilde X_{T^c_t}$. Then, by the same argument in the proof of Theorem \ref{noarbitrage}, we can conclude that $e^Y$ has the no arbitrage property with respect to $\Pi((\scG_{T^c_t})_{t\geq 0})$. Since $Y_t$, $T^c_t$, and $T^j_t$ are $\scG_{T^c_t}$-measurable for all $t$, $\bbF^{Y, T^c,T^j}\subset(\scG_{T^c_t})_{t\geq 0}$. Therefore, applying Theorems 2 and 3 of JPS proves the claim that $e^Y$ has the no arbitrage property with respect to $\Pi(\bbF^{Y, T^c, T^j})$.
	\end{proof}
	
	\subsection{Hedging contingent claims}
	The natural follow-up question is whether we can apply our model to hedging contingent claims. This is an essential step in pricing contingent claims, but it is not obvious since our asset price $S$ is not necessarily a semimartingale. The non-semimartingale property does not permit the existence of an equivalent local martingale measure, so we cannot use the Fundamental Theorems of Asset Pricing. A few papers examine the hedging problem when the asset price is not a semimartingale, such as those by \cite{biagini2003minimal} and \cite{guasoni2008consistent}. However, it is not obvious how we can apply those results to our model.
	
	In addition, we restricted our hedging strategies to the Cheridito Class, which is a strict subset of the space of predictable processes, in order to define stochastic integrals with respect to $S$. Though we obtained the no arbitrage property in the previous subsection, we cannot be sure whether the market is complete. 
	
	In this section, we briefly explain some solutions. We resolve the issue by approximating $S$ with a semimartingale, proving the existence of an equivalent martingale measure, and choosing an equivalent martingale measure that satisfies specific criteria.
	
	We consider the conditions in Theorem \ref{thm:noarbittcj}. Let the filtration $\bbF$ be $\bbF^{Y,T^c,T^j}$, as in Theorem \ref{thm:noarbittcj}. Let $(\tau_k)_{k=0}^\infty$ be an increasing sequence of $\bbF$-stopping times with $\tau_0=0$, and $\tau_n$ is bounded almost surely. We assume that for a given $\veps>0$, $\tau_{i+1}-\tau_i>\veps$ almost surely for all $i$. Then, we approximate $S$ with a pure jump process: 
	\begin{align*}
	\tilde S :=\sum_{i=0}^n S_{\tau_i}{\bf 1}_{[\tau_i,\tau_{i+1})}.
	\end{align*}
	\begin{lemma}
		$\tilde S$ is a $\bbF$-semimartingale.
	\end{lemma}
	\begin{proof}
		Note that the number of jumps on the time interval $[0,M]$ is almost surely finite for any given $M$. Moreover, $\sup_{t\in[0,M]}|S_t|<\infty$ almost surely. Therefore, $\tilde S$ has a finite variation almost surely on $[0,M]$ for any $M$. Then, by Theorem II.3.7 of \cite{Protter:2004wf}, we can conclude that $\tilde S$ is a semimartingale.
	\end{proof}
	\begin{Example}
		In a real market, the price of an asset has discrete values determined by tick size, which is the minimum price increment. We can define $\tau_i$ as the $i$th price movement of the asset. Then, realistic restriction should make the price changes at least $\veps$ milliseconds apart. Therefore, approximating $S$ with $\tilde S$ is not a very restrictive assumption in modeling the real market.
	\end{Example}
	
	\begin{lemma}
		Let $\bar \bbF:=(\bar\scF_t)_{t\geq 0}$, where $\bar\scF_t=\scF^{\tilde S}_t$. The semimartingale $\tilde S$ does not allow arbitrage for $\bar\bbF$-predictable trading strategies. Moreover, there exists an equivalent martingale measure (risk neutral measure).
	\end{lemma}
	\begin{proof}
		Note that since we are considering the pure jump process $\tilde S$ and its natural filtration, our model is essentially discrete. Assume that there is an arbitrage; then we have $H=\sum_{j=0}^\infty H_{\tau_j}{\bf 1}_{(\tau_j,\tau_{j+1}]}$ such that
		$\int_0^\infty H_t d\tilde S_t \geq 0$ almost surely and $\int_0^\infty H_td\tilde S_t>0$ with a positive probability. Since $\tau_{j+1}-\tau_{j}>\veps$ for all $j$, we have $H\in\Pi(\bar\bbF)$. Since
		\begin{align*}
		\int_0^\infty H_tdS_t=\sum_{j=0}^\infty H_{\tau_j}\brak{S_{\tau_{j+1}}-S_{\tau_j}}=\int_0^\infty H_t d\tilde S_t,
		\end{align*}
		the claim is in contradiction to Theorem \ref{thm:noarbittcj}. Therefore, $\tilde S$ has the no arbitrage property.
		Since we can consider our market as a time-changed discrete model, we can apply the Dalang-Morton-Willinger Theorem (\cite{dalang1990equivalent}) to conclude there is an equivalent martingale measure.
		%
		%
		%
		%
	\end{proof}
	
	Since we have an equivalent martingale measure, we can choose one of the equivalent martingale measures that satisfies a specific set of conditions. Several criteria have been proposed: minimal variance (\cite{delbaen1996variance,schweizer1995variance}), minimal martingale (\cite{follmer1991hedging, schweizer1995minimal}), minimal entropy martingale (\cite{frittelli2000minimal}), and so on. The choice of specific martingale measure will induce the corresponding hedging strategy $H^0$, as well as the price of the contingent claim. The hedging strategy will be a predictable process in general, and we verify that it can be approximated by a hedging strategy in the Cheridito Class. 
	
	\begin{definition}
		Let $X$ be a special semimartingale with canonical decomposition $X=N+A$, where $N$ is a local martingale and $A$ is a predictable finite variation process. The $\scH^2$ norm of $X$ is
		\begin{align*}
		\norm{X}_{\scH^2}=\norm{[N,N]^\half_\infty}_{L^2}+\norm{\int_0^\infty|dA_s|}_{L^2}.
		\end{align*}
		A $\scH^2$-semimartingale is a special semimartingale with a finite $\scH^2$ norm.
	\end{definition}
	\begin{theorem}
		Let $H^0$ be a predictable process that is $\tilde S$ integrable. Then, there exists a sequence $(H^n)_{n=1}^\infty\subset\Pi(\bbF^{\tilde S})$ such that $H^n$ are $\tilde S$ integrable and
		\begin{align*}
		\norm{\int_0^\infty H^0_td\tilde S_t-\int_0^\infty H^n_td\tilde S_t}_{\scH^2}\xrightarrow{n\to\infty}0.
		\end{align*}
	\end{theorem}
	\begin{proof}
		Without loss of generality, we can assume $\tilde S$ is a $\scH^2$-semimartingale and $H^0$ is a bounded predictable process by appropriate localization: see Theorem IV.2.13 and Theorem IV.2.14 of \cite{Protter:2004wf}.
		
		Due to Lemma 7 of \cite{jarrow2009no}, the Cheridito Class is dense in $b\bbL$, the space of adapted processes with bounded c\`agl\`ad paths, in the UCP topology, where UCP denotes uniform convergence in probability on compact time sets. 
		
		For $\tilde S=N+A$, where $N$ and $A$ are a local martingale and a predictable finite variation process, respectively, we define metric $d$ by
		\begin{align*}
		d(H,J):=\norm{\brak{\int_0^\infty (H-J)^2d[N,N]_\scT}^{\half}}_{L^2}+\norm{\int_0^\infty |H-J||dA_s|}_{L^2}.
		\end{align*}
		Then, by Theorem IV.2.2 of \cite{Protter:2004wf}, $b\bbL$ is dense in the space of bounded predictable processes under $d$. Since the stochastic integral of a predictable process with respect to $\tilde S$ is defined as a continuous extension of a function from $(b\bbL,d)$ to $\scH^2$, there exists a sequence $(H^n)_{n=1}^\infty\subset\Pi(\bbF^{\tilde S})$ that converges to $H^0$ in $d$, and thus
		\begin{align*}
		\norm{\int_0^\infty H^0_td\tilde S_t-\int_0^\infty H^n_td\tilde S_t}_{\scH^2}\xrightarrow{n\to\infty}0.
		\end{align*}
	\end{proof}

	\section{The distribution and moment-generating function of the correlated TCL process}
	\kihun{The transitional density of the correlated TCL process $Y=X_{T}$ is not available in closed form unless $T$ and $X$ are independent. When the business time $ T_t $ is independent of $ X_t $, the characteristic function of $Y_t$ is just the Laplace transform of $ T_t $ evaluated at the cumulant exponent $ \Psi_{X}(\theta) :=-\frac{1}{t}\log\bbE e^{i\theta X_t}$ of $X$ by conditioning on $T_t$. On the other hand,
		when the \levy process $ X $ and stochastic time process $ T$ are dependent, it is not obvious how to compute the characteristic function of $ Y_t $. 
		Our primary objective in this section is to find formulas, as explicitly as possible, for the PDF, the characteristic function, and the moment-generating function of $ Y_t $.}
	
	\subsection{Distribution of the correlated TCL process}
	First, assume $\alpha=0$ and $\beta=1$ in \eqref{levy_process} such that $Y_t=Z_{J_{T_t}}$. In order to find the distribution of $Z_{J_{T_t}}$, one can either try to calculate the PDF or an integral transform of it. In this section, we discuss the calculation of the probability density distribution directly.
	
	We can express the PDF of a TCL process by the joint distribution of the time change and Brownian motion that drives the activity rate.
	\begin{proposition}\label{mainthm1}
		Under the assumptions in the previous section, we have the following formula
		\begin{align}\label{pdf_tclp}
		\quad f_{Z_{J_{T_t}}}(z)&=\iiint \frac{1}{\sqrt{2\pi j(1-\rho^2)}}e^{-\frac{(z-\rho b)^2}{2(1-\rho^2)j}} f_{B_j,T_t}( b,y)f_{J_y}(j)dyd bdj.
		\end{align}
	\end{proposition}
	Therefore, to calculate the probability density distribution of $Z_{J_{T_t}}$, one needs to calculate the joint distribution function $ B_j $ and $ T_t $. When $j\leq t$, $f_{B_j, T_t}(b,y)=\int f_{T_t|B_j, T_j}(y|b,\bar y)f_{B_j,T_j}(b,\bar y)d\bar y$. In this case, using the Markovian property of $(v,T,B)$, we can calculate $f_{T_t|B_j, T_j}$.
	On the other hand, if $j\geq t$, \begin{align}
	f_{B_j, T_t}(b,y)&=\int f_{B_j|B_t, T_t}(b|z,y)f_{B_t, T_t}(z,y)dz\\
	&=\int f_{B_j-B_t}(b-z)f_{B_t, T_t}(z,y)db=(f_{B_j-B_t}*f_{B_t, T_t}(\cdot,y))(b).
	\end{align}
	
	Therefore, we need only to calculate the joint PDF of $B_t$ and $T_t$. Unfortunately, even when $(v,T)$ is an affine process, $(v,T,B)$ is not affine, and therefore, a closed-form representation of the characteristic function is not available. Alternatively, we can use the Fokker-Plank PDE. For $t\leq j$, let $q:(t,x,y,z)\to q(t,x,y,z)$ be the unique solution to the following Fokker-Plank PDE for $(v_t, T_t,B_t)$:\footnote{The existence of a unique solution is guaranteed by the existence of a unique solution for \eqref{sde}}
	\begin{align}\label{FPq}
	\partial_t q&=\half \partial_{xx}[\sigma^2q]+\partial_{xz}[\sigma q]+\half \partial_{zz}q-\partial_x[\mu q]-x\partial_yq\\
	q(0,x,y,z)&=\delta_{v_0}(x)\delta(y)\delta(z).
	\end{align}
	\begin{remark}
		Assume that $\mu(t,x)=\kappa(1-x)$ and $\sigma(t,x)=\sigma_v\sqrt{x}$. In this case, $v$ follows a CIR process and $(v,T)$ forms a two-dimensional affine process. Though we have no closed-form solution to PDE \eqref{FPq}, we know that the PDF of $(v_t, T_t-t,B_t/\sqrt{t})$ converges asymptotically as $t\to\infty$. The following elliptical PDE gives this limit:
		\begin{align}
		\half \partial_{xx}[\sigma_v^2x\tilde q(x,y,z)]+\partial_{xz}\edg{\sigma_v\sqrt{x/t} \tilde q(x,y,z)}+\half \partial_{zz}\edg{\tilde q(x,y,z)/t}=0.
		\end{align}
		To compute \eqref{FPq} numerically, one may approximate $q(t,x,y,z)$ with $\tilde q(t,x,y+t,\sqrt{t}z)$ when $t$ is large enough.
	\end{remark}
	
	Given the Fokker-Planck PDE \eqref{FPq}, we can calculate the joint distribution of $ T_t $ and $ B_j $, which is given by the following proposition.
	\begin{proposition}\label{prep_joint1}
		Let $q$ be the unique solution to \eqref{FPq}. Then,
		\begin{align}
		f_{T_t,B_j}(y,z)=\begin{cases}
		\frac{1}{\sqrt{2\pi(j-t)}}\iint e^{-\frac{(z-\tilde z)^2}{2(j-t)}}q(t,x,y,\tilde z) dxd\tilde z,&\text{ $t\leq j$}\\
		\frac{1}{2\pi}\iiint \exp\brak{\Phi(t,j,\xi,x,\tilde y)-i\xi y}q(j,x,\tilde y,z)dxd\tilde yd\xi,&\text{ $t\geq j$}.
		\end{cases}
		\end{align}
	\end{proposition}
	The Fokker-Plank PDE \eqref{FPq} is three dimensional, but it can be reduced to a one-dimensional PDE if we use Fourier transform for variables $y$ and $z$. This dimensional reduction is useful in numerical computations. In other words, the Fourier transform of \eqref{FPq} is
	\begin{align}\label{FPqFourier}
	\partial_t \hat q&=\half \partial_{xx}[\sigma^2\hat q]-\partial_x[(\mu-i\eta\sigma) \hat q]-\brak{\frac{\eta^2}{2}+ix\xi}\hat q\\
	\hat q(0,x,\xi,\eta)&=\delta_{v_0}(x).
	\end{align}
	Using $\hat q$, the PDF of $Z_{J_{T_t}}$ is given by the following theorem.
	\begin{theorem}\label{mainthm0}
		Let $\hat q$ be the solution to \eqref{FPqFourier}. Then,
		\begin{align*}
		f_{Z_{J_{T_t}}}(z)&=\frac{1}{(2\pi)^2\rho}\iiint \hat f_{T_t,B_j}(\xi,\eta)e^{\frac{i\eta z}{\rho}-\frac{(1-\rho^2)\eta^2j}{2\rho^2}}\brak{\int e^{i\xi y}f_{J_y}(j)dy} d\xi d\eta dj,
		\end{align*}
		where
		\begin{align}
		\hat f_{T_t,B_j}(\xi,\eta)=\begin{cases}
		e^{-\frac{(j-t)\eta^2}{2}}\int\hat q(t,x,\xi,\eta) dx,&\text{ $t\leq j$}\\
		\frac{1}{2\pi}\iint \brak{\int \exp\brak{\Phi(t,j,-\xi,x,y)+i\bar \xi y}dy}\hat q(j,x,\bar{\xi},\eta)d\bar\xi dx,&\text{ $t\geq j$.}
		\end{cases}
		\end{align}
	\end{theorem}
	\begin{remark}
		When $(v,T)$ has a closed-form solution,\footnote{For example, assume $ \mu(t,x)=\kappa(1-x)$ and $\sigma(t,x)=\sigma_v x$. Then, we have a closed-form solution $(v_t,T_t)$, where
			\begin{align}
			v_t&=e^{-\kappa t-\sigma_v^2 t/2+\sigma_v B_t}\brak{1+\kappa\int_0^te^{\kappa s+\sigma_v^2s/2-\sigma_v B_s}ds},\qquad
			T_t=\int_0^tv_udu.
			\end{align}} then, we can express $q(t,x,y,z)$ using a Brownian bridge from $0$ to $z$.
		We can define the path functional $\scV,\scW:C[0,t]\to\bbR$ such that $v_t=\scV(B_{[0,t]})$ and $T_t=\scW(B_{[0,t]})$.
		Then, $ f_{v_t,T_t|B_t}(x,y|z)=f_{\tilde v_t,\tilde T_t}(x,y) $,
		where $\tilde v_t=\scV(\hat B_{[0,t]}), \tilde T_t=\scW(\hat B_{[0,t]})$, and $\hat B$ is a Brownian bridge from $0$ to $z$ on time interval $[0,t]$; that is,
		\begin{align}
		\hat B_s= \tilde B_s-\frac{s(\tilde B_t-z)}{t},
		\end{align}
		where $\tilde B$ is a Brownian motion. Then,
		\begin{align}
		q(t,x,y,z)=\frac{1}{\sqrt{2\pi t}}e^{-\frac{z^2}{2t}}f_{\tilde v_t,\tilde T_t}(x,y).
		\end{align}
	\end{remark}
	For a general $Y$ of the form \eqref{levy_process}, we can easily extend the previous theorem.
	\begin{theorem}\label{main_general}
		\begin{align}
		f_{Y_t}(\tilde y)=\frac{\sqrt{1-\rho^2}}{(2\pi)^2|\rho|}\iiint \hat f_{T_t,B_j}(\xi,\eta)e^{\frac{i\eta (\tilde y-\alpha j)}{\rho\beta}-\frac{(1-\rho^2)\eta^2j}{2\rho^2}}\brak{\int e^{i\xi y}f_{J_y}(j)dy} d\xi d\eta dj,
		\end{align}
		where
		\begin{align}
		\hat f_{T_t,B_j}(\xi,\eta)=\begin{cases}
		e^{-\frac{(j-t)\eta^2}{2}}\int\hat q(t,x,\xi,\eta) dx,&\text{ $t\leq j$}\\
		\frac{1}{2\pi}\iint \brak{\int \exp\brak{\Phi(t,j,-\xi,x,y)+i\bar \xi y}dy}\hat q(j,x,\bar{\xi},\eta)d\bar\xi dx,&\text{ $t\geq j$.}
		\end{cases}
		\end{align}
	\end{theorem}
	\subsection{The moment-generating function of the correlated TCL process}\label{sec_mgf_cf}
	It is noteworthy that the Laplace transform of $Z_{J_{T_t}}$ can be factored into the characteristic functions of $W$, the inverse of $J$, and the measure-changed $T$. Let $v^{j,r}_0=1$, $T^{j,r}_0=0$, and
	\begin{align}\label{sdechanged}
	dv^{(j,r)}_u&=\brak{\mu(u,v^{(j,r)}_u)-r\rho \sigma(u,v^{(j,r)}_u){\bf 1}_{u\leq j}}du+\sigma(u,v^{(j,r)}_u)dB_u\\
	dT^{(j,r)}_u&=v^{(j,r)}_udu.
	\end{align}
	Note that $(v^{(j,r)},T^{(j,r)})$ is the weak solution to \eqref{sde} under a measure of $\bbQ$ that makes $B_t-r\rho{\bf 1}_{[0,j]}(t)$ a Brownian motion. We have the following theorem.
	\begin{theorem}\label{mainthm} 
		For $T^{(j,r)}$ given by SDE \eqref{sdechanged}, we have
		\begin{align}\label{mainformula}
		\bbE e^{-rZ_{J_{T_t}}}=\frac{1}{2\pi}\iint e^{r^2j/2}\brak{\int e^{i\xi y}f_{J_y}( j)dy}\bbE\edg{\exp\brak{-i\xi T^{(j,r)}_t}}d\xi dj.
		\end{align}
	\end{theorem}
	\begin{remark}
		Note that $e^{r^2j/2}$ is the Laplace transform of the Brownian motion at time $j$.
	\end{remark}
	\begin{remark}\label{mainrmk}
		For the characteristic function of $Z_{J_{T_t}}$, we have the following formula:
		\begin{align}
		\bbE e^{i\theta Z_{N_{Y_t}}}=\frac{1}{2\pi}\iint e^{-\theta^2j/2}\brak{\int f_{J_y}( j)e^{i\xi y}dy}\bbE\edg{e^{-i\xi T_t}\scE^{\theta,j}_t}d\xi dj,
		\end{align}
		where
		\begin{align}
		\scE^{\theta,j}_t:=\exp\brak{i\theta\rho B_{j\wedge t}+\frac{\theta^2\rho^2(j\wedge t)}{2}}.
		\end{align}
		Here, note that $e^{-\theta^2j/2}$ is the characteristic function of the Brownian motion at time $j$.
		Heuristically, $\bbE\edg{e^{-i\xi T_t}\scE^{\theta,j}_t}$ is the expectation of $e^{-i\xi T_t}$ under the complex measure given by $\scE^{\theta,j}_td\bbP$. However, a rigorous validation of such statement is tricky (see Warning 1.7.3.2 of \cite{JEANBLANC:2009wy}). The Appendix provides the actual derivation.
	\end{remark}
	Theorem \ref{mainthm} implies that one can transform the dependency of $Z$ and $T$ in terms of the measure change; note that the distribution of $T$ under $\bbP$ is the same as the distribution of $T^{(j,r)}$ under $\bbQ$ and $T^{(j,r)}=T$ if $\rho=0$. In this sense, Theorem \ref{mainthm} is analogous to Theorem 1 of CW, which finds a (complex) measure such that one can calculate the characteristic function of the TCL as the time change and the \levy process is independent.
	
	Assuming that $J$ is strictly increasing, there exists the inverse $F(\omega)$ of $J(\omega)$ for each $\omega\in\Omega$, and then we can express the Laplace transform in terms of the characteristic functions of $W,F,$ and $T^{(j,r)}$.
	\begin{corollary}
		In addition to the conditions in Theorem \ref{mainthm}, assume that $J$ is a strictly increasing process and define $F_j$ such that $F_j(\omega)$ is the inverse of $J_y(\omega)$ for each $\omega\in\Omega$. We define $\chi(a,X):=\bbE e^{iaX}$ as the characteristic function for any random variable $X$. Then,
		\begin{align}\label{corformula}
		\bbE e^{-rZ_{J_{T_t}}}=\frac{1}{2\pi}\iint \chi\brak{ir, W_{j}}\chi\brak{\xi, F_{j}}\chi\brak{-\xi, T^{(j,r)}_t}d\xi dj.
		\end{align}
	\end{corollary}
	
	Note that in Theorem \ref{mainthm}, the only ambiguous term is $\bbE\edg{\exp\brak{-i\xi T^{(j,r)}_t}}$, which we can calculate in a closed form for only highly restricted cases.
	In general, we need to solve a Fokker-Plank PDE to calculate it. Note that the PDF of $(v^{(j,r)}_t,T^{(j,r)}_t)$ under $\bbP$ does not depend on $j$ if $t\leq j$. Let the PDF of $(v^{(j,r)}_t,T^{(j,r)}_t)$ be
	$G(t,x,y)$ when $t\leq j$. Then, $G$ is a (weak) solution to the following Fokker-Plank PDE:
	\begin{align}\label{FKG}
	(\partial_u G)(u,x,y)&=\half\partial_{xx}\edg{|\sigma(u,x)|^2G(u,x,y)}-\partial_{x}\edg{\brak{\mu(u,x)-r\rho\sigma(u,x)}G(u,x,y)}\\
	&\quad-x\partial_yG(u,x,y)\\
	G(0,x,y)
	&=\delta_{v_0}(x)\delta(y).
	\end{align}
	Hence, we can express the Fourier transform of $T^{(j,r)}_t$ using $G$.
	\begin{proposition}\label{propG} If \eqref{FKG} has a unique solution $G$, then
		\begin{align}
		\bbE\edg{\exp\brak{-i\xi T^{(j,r)}_t}}=\begin{cases}
		\iint\exp(\Phi(t,j,-\xi,x,y))G(j,x,y)dxdy&\text {if $t>j$}\\
		\iint \exp\brak{-i\xi y}G(t,x,y)dxdy&\text {if $t\leq j$}.
		\end{cases}
		\end{align}
	\end{proposition}
	
	It is noteworthy that if $\Phi(t,s,\xi,x,y)$ is affine with respect to $y$, then the Fokker-Plank PDE can be reduced to one dimension using Fourier transform.
	\begin{proposition}\label{propFTG}
		Assume that
		\begin{align}
		\Phi(t,s,\xi,x,y)=\phi(t,s,\xi,x)+i\psi(t,s,\xi,x)y
		\end{align} and that there is a unique solution to
		\begin{align}\label{FTG}
		(\partial_u \hat G)(u,x,\eta)&=\half\partial_{xx}\edg{|\sigma(u,x)|^2\hat G(u,x,\eta)}-\partial_{x}\edg{\brak{\mu(u,x)-r\rho\sigma(u,x)}\hat G(u,x,\eta)}\\
		&\quad-i\eta x\hat G(u,x,\eta).
		\end{align}
		with $\hat G(0,x,\eta)=\delta_{v_0}(x)$.
		In this case,
		\begin{align}
		\bbE\edg{\exp\brak{-i\xi T^{(j,r)}_t}}=\begin{cases}
		\int e^{\phi(t,j,-\xi,x)} \hat G(j,x,-\psi(t,j,-\xi,x))dx &\text{ for $t>j$}\\
		\int \hat G(t,x,\xi) dx&\text{ for $t\leq j$}
		\end{cases}.
		\end{align}
		
	\end{proposition}
	
	\begin{remark}
		For the characteristic function $\bbE[e^{i\theta Z_{J_{T_t}}}]$, we have a formula that corresponds to the case in which $r=-i\theta$; that is,
		\begin{align}
		\bbE e^{i\theta Z_{J_{T_t}}}&=\frac{1}{2\pi}\iint e^{-\theta^2j/2}\brak{\int f_{J_y}( j)e^{i\xi y}dy}\bbE\edg{e^{-i\xi T_t}\scE^{\theta,j}_t}d\xi dj\\
		\bbE\edg{e^{-i\xi T_t}\scE^{\theta,j}_t}&=\begin{cases}
		\iint\exp\brak{\Phi(t,j,-\xi,x,y)}G^\theta(j,x,y)dxdy&\text {if $t>j$}\\
		\iint \exp\brak{-i\xi y}G^\theta(t,x,y)dxdy&\text {if $t\leq j$},
		\end{cases}
		\end{align}
		where $G^\theta$ is the solution to
		\begin{align}
		(\partial_u G^\theta)(u,x,y)&=\half\partial_{xx}\edg{|\sigma(u,x)|^2G^\theta(u,x,y)}-\partial_{x}\edg{\brak{\mu(u,x)+i\theta\rho\sigma(u,x)}G^\theta(u,x,y)}\\\notag
		&\quad-x\partial_yG^\theta(u,x,y)\\
		G^\theta(0,x,y)
		&=\delta_{v_0}(x)\delta_{T_0}(y).
		\end{align}
		As we noted earlier, we cannot simply change $r=-i\theta$ to extend the Girsanov transform. See the appendix for the derivation of and details about the statement above.
	\end{remark}
	Let $ Y_t=\alpha J_{T_t}+\beta Z_{J_{T_t}} $. The following theorem gives the moment-generating function for the TCL process $Y$.
	\begin{theorem}\label{mgf_generalmodel} Let $Y_t:=\alpha J_{T_t}+\beta Z_{J_{T_t}}$. Then, the Laplace transform of $Y$ is
		\begin{align}
		\bbE e^{-rY_t}=\frac{1}{2\pi}\iint e^{r^2\beta^2j/2-r\alpha j}\brak{\int e^{i\xi y}f_{J_y}( j)dy}\bbE\edg{\exp\brak{-i\xi T^{(j,r\beta)}_t}}d\xi dj.
		\end{align}
	\end{theorem}
	\begin{remark}
		Note that for a constant $j$, the Laplace transform of $\alpha j+\beta W_{j}$ is $e^{r^2\beta^2j/2-r\alpha j}$.
	\end{remark}

	\section{\hasan{Model specifications} }\label{HW}
	
	\hasan{We develop candidate option pricing models by modeling the underlying asset return process as the sum of two TCL processes. Our framework is based on two ingredients: the jump structure and the source and dynamic of the stochastic volatility component(s). We show that all specified option pricing models in \cite{Huang:in} are nested to our framework.}
	
	
	\subsection{Dynamic of underlying asset returns}\label{sec_dynamic_returns}
	
	\hasan{We use a correlated TCL process $ Y$ to model economic uncertainty. The price of the financial asset (e.g., index) under consideration evolves in continuous time and is given by 
		\begin{align}
		S_t=S_0 e^{(r_{int}-\delta_{div})t+Y_t},
		\end{align}
		where $r_{int}$ and $\delta_{div}$ are constants that represent, respectively, the risk free rate of interest and the dividend yield of the asset, and 
		\begin{align}
		Y_t=\log\left(e^{\delta_{div} t} S_t/S_0\right)-r_{int}t
		\end{align}
		models excess log returns, under the usual assumption that dividends are continuously reinvested. }
	
	\hasan{In our framework, for a given filtration $\bbF$, the process $X_t$ is an adapted \levy process. Then, by the L\'evy-It\^o decomposition theorem, there are continuous \levy $X^c$ and pure-jump \levy processes $X^j$ that are independent and $X=X^c+X^j$. Along the line of \cite{Huang:in}, we study the effect of time changes $T^c$ and $T^j$, which we apply to $X^c$ and $X^j$, respectively. We can assume the following without losing generality:
		\begin{align}\label{corr_TCLP}
		Y_t&=X^c_{T^c_t}+X^j_{T^j_t}\\
		X^c_t&=a^c_1t+a^c_2B^c_t+a^c_3B^j_t+a^c_4 W^c_t\\
		X^j_t&= a^j_1J_t+a^j_2W^j_{J_t}\\
		dv^c_t&=\mu^c(1-v^c_t)dt+\sigma^c\sqrt{v^c_t(1-\rho^2)}dB^c_t+\rho\sigma^c\sqrt{v^c_t}dB^j_t\\
		dv^j_t&=\mu^j(1-v^j_t)dt+\sigma^j\sqrt{v^j_t}dB^j_t\\
		T^c_t&=\int_0^tv^c_sds\\
		T^j_t&=\int_0^tv^j_sds.
		\end{align}
		Here, $(a^x_i,\mu^x,\sigma^x)_{x\in\crl{c,j}, i=1,2,...}$, and $\rho$
		are given constants, $B^c,B^j,W^c, W^j$ are independent Brownian motions, and $J$ is a pure jump process. This is possible because, under the adaptedness condition in \cite{Huang:in}, $X^j$ is always independent of Brownian motions $B^c,B^j$, and $W^c$ by Remark \ref{propindep}.
	}
	
	\subsection{Jump structure}
	
	Any \levy jump process (or any \levy process) is uniquely defined by its \levy triplet: location, scale, and \levy measure. Its \levy measure captures the structure of the jump, which controls the arrival rate of jumps of size $ x\in\mathbb{R}^0 $ (the real line excluding zero). We can group \levy jump processes into three categories: (i) finite activity, (ii) infinite activity with finite variation, and (iii) infinite activity with infinite variation.
	
	Here, we define a \levy process by changing the time of a Brownian motion with a subordinator, which is a pure jump process; that is, equation \eqref{levy_process}. The advantage of this definition is that one can construct any jump process from \eqref{levy_process} by specifying only the subordinator $ J_t $. For example, if $ J_t $ is the gamma process, then $ X_t^j $ is the variance gamma process. One can show that equation \eqref{levy_process} encompasses all specifications for the jump structure in \cite{Huang:in}.

	\subsection{Source of stochastic volatility}
	
	\hasan{Equation \eqref{corr_TCLP} shows that there are several ways to introduce stochastic volatility for a correlated TCL process. The stochastic volatility can come either from the instantaneous variance of the continuous component $ X^c_t $, from the arrival rate of the jump component $ X^j_t $, or both. 
		
		\cite{Huang:in} consider the return process before time changes as the sum of a Brownian motion with constant drift and a pure jump process.\footnote{We can express any \levy process in this way via the L\'evy-It\^o decomposition, see equation (3) of \cite{Huang:in}.} They specify four different stochastic volatility models by changing the time of the Brownian motion and/or jump component of the return process. Their four stochastic volatility models are: stochastic volatility from a diffusion component (SV1); stochastic volatility from a jump component (SV2); joint contribution from jump and diffusion (SV3), where the stochastic time change is the same for both the diffusion and jump component; and joint contribution from jump and diffusion processes with different stochastic time changes (SV4).
		
		If we apply a stochastic time change to the continuous component only, we have $ Y_t=X^c_{T^c_t}+X^j_t $. The arrival rate of jumps remains constant. This case is equivalent to the SV1 models in \cite{Huang:in}. \cite{bates1996jumps} and \cite{bakshi1997empirical} are also nested in this setting. Alternatively, if we leave the continuous component unchanged and apply a stochastic time change only to the jump component; that is, $ Y_t=X^c_{t}+X^j_{T^j_t} $, then stochastic volatility comes from just the time variation in the arrival rate of jumps. This case is equivalent to the SV2 models in \cite{Huang:in}. Another possibility for introducing stochastic volatility to the model is that stochastic volatility comes from both the continuous and jump components simultaneously; that is, $ Y_t=X^c_{T^c_t}+X^j_{T^c_t} $. This case is equivalent to the SV3 models in \cite{Huang:in}. \cite{bates2000post} and \cite{pan2002jump} are also variations of this category. The last alternative is to apply separate time changes to the continuous and jump components. Under this specification, the instantaneous variance of the continuous component and the arrival rate of the jump component follow separate stochastic processes. This case is analogous to the SV4 models in \cite{Huang:in}. We should note that SV4 encompasses all of the other three specifications (SV1--SV3) as special cases. 
	}

	\subsection{Moment generating function}
	\hasan{We devote this section to deriving the moment-generating function for process $ Y_t=X^c_{T^c_t}+X^j_{T^j_t} $, where $ X^c_{T^c_t} $ and $ X^j_{T^j_t} $ represents its continuous and jump part, respectively.\footnote{To save space, we do not discuss the derivation of the characteristic function; however, one can obtain the characteristic function by deploying a similar technique, see Section \ref{sec_mgf_cf}.} Under the assumptions of Section \ref{sec_dynamic_returns}, one can calculate the Laplace transform of $ Y_t$,
		for specific choices of constants that correspond to the stochastic volatility models of \cite{Huang:in}. For notational convenience, we denote $ \scL_J(r, t):=\bbE e^{-r J_t} $. Note that $J$ is given by our model and
		\begin{align}
		\bbE \edg{\exp\brak{-rX^j_{T^j_t}}\vert T^j_t}&=\bbE\edg{ \exp\brak{-ra^j_1 J_{T^j_t}-ra^j_2W^j_{J_{T^j_t}}}\vert T^j_t}\\
		&=\bbE\edg{ \exp\brak{-ra^j_1 J_{T^j_t}}\bbE\edg{\exp\brak{-ra^j_2W^j_{J_{T^j_t}}}\vert J_{T^j_t}, T^j_t}\vert T^j_t}\\
		&=\bbE \edg{ \exp\brak{-ra^j_1 J_{T^j_t}+\frac{r^2|a^j_2|^2}{2}J_{T^j_t}\vert T^j_t}}\\
		&=\scL_{J}\brak{ra^j_1-\frac{r^2|a^j_2|^2}{2}, T^j_t}.
		\end{align}
		In addition,
		\begin{align}
		\bbE \exp\brak{-rX^c_{T^c_t}}&=\bbE \exp\brak{-ra^c_1 T^c_t+a^c_2B^c_{T^c_t}+a^c_3B^j_{T^c_t}+a^c_4 W^c_{T^c_t}}\\
		&=\bbE\edg{ \exp\brak{-r\brak{a^c_1 T^c_t+a^c_2B^c_{T^c_t}+a^c_3B^j_{T^c_t}}}\bbE\edg{\exp\brak{-ra^c_4 W^c_{T^c_t}}\vert B^c, B^j}}\\
		&=\bbE\edg{ \exp\brak{-r\brak{(a^c_1-r|a^c_4|^2/2) T^c_t+a^c_2B^c_{T^c_t}+a^c_3B^j_{T^c_t}}}}.
		\end{align}
		Note that for $N:=|(\rho a^c_2-\sqrt{1-\rho^2}a^c_3,(1-\rho^2)a^c_3-\rho\sqrt{1-\rho^2}a^c_2)|, B:=\sqrt{1-\rho^2}B^c+\rho B^j$, 
		\begin{align}
		\tilde B:=\frac{1}{N}\brak{(\rho^2a^c_2-\rho\sqrt{1-\rho^2}a^c_3)B^c+((1-\rho^2)a^c_3-\rho\sqrt{1-\rho^2}a^c_2)B^j},
		\end{align}
		we know $B$ and $\tilde B$ are independent Brownian motions with
		\begin{align}
		\hat{B}_t&:=(a^c_1-r|a^c_4|^2/2)t+a^c_2B_t^c+a^c_3B_t^j\\
		&= (a^c_1-r|a^c_4|^2/2)t+(a^c_2\sqrt{1-\rho^2}+\rho a^c_3) B_t+N\tilde B_t.
		\end{align}
		The Laplace transform of $X^c_{T^c_t}=\hat{B}_{T^c_t}$ is
		\begin{align}
		\bbE \exp(-r \hat{B}_{T^c_t})&=\bbE\edg{\exp\brak{-r(a^c_2\sqrt{1-\rho^2}+\rho a^c_3) B_{T^c_t}-r(a^c_1-r|a^c_4|^2/2)T^c_t}\bbE\edg{e^{-rN\tilde B_{T^c_t}}\vert B}}\\
		&=\bbE\edg{\exp\brak{-r(a^c_2\sqrt{1-\rho^2}+\rho a^c_3) B_{T^c_t}-r(a^c_1-r(|a^c_4|^2+N^2)/2)T^c_t}}
		\end{align}
		If we let $J_t=t, \alpha=a^c_1-r(|a^c_4|^2+N^2)/2,\beta=a^c_2\sqrt{1-\rho^2}+\rho a^c_3$, and $Z=B$ in Theorem \ref{mgf_generalmodel}, we obtain
		\begin{align}
		\bbE e^{-rX^c_{T^c_t}}=\frac{1}{2\pi}\iint e^{r^2\beta^2j/2-r\alpha j+i\xi j}\bbE\edg{\exp\brak{-i\xi \tilde T_t}}d\xi dj,
		\end{align}
		where
		\begin{align}\label{tildeT}
		d\tilde v_t&=\edg{\mu^c(1-\tilde v_u)-r\beta\sigma^c \sqrt{\tilde v_u}{\bf 1}_{u\leq j}}du+\sigma^c\sqrt{\tilde v_u}dB_u,\\
		\tilde T_t&:=\int_0^t\tilde v_udu.
		\end{align}
		
		\subsubsection{When $T^j$ and $(X^c, T^c)$ are independent}
		
		The case in which $T^j$ is independent of $X^c$ and $T^c$ includes the SV1 ($v^j_0=1, \mu^j=\sigma^j=\rho=0$) and SV4 ($\rho=0$) models in \cite{Huang:in}. In this case, we have $\rho=0$ in the previous subsection and 
		\begin{align}
		\bbE e^{-rY_t}&=\bbE \exp\brak{-rX^c_{T^c_t}}\bbE \exp\brak{-r X^j_{T^j_t}}\\
		&=\frac{1}{2\pi}\bbE\edg{\scL_{J}\brak{ra^j_1-\frac{r^2|a^j_2|^2}{2}, T^j_t}}\iint e^{r^2\beta^2j/2-r\alpha j+i\xi j}\bbE\edg{\exp\brak{-i\xi \tilde T_t}}d\xi dj,
		\end{align}
		where $\tilde T$ is given by \eqref{tildeT}.
		
		\subsubsection{When $T^c$ and $(X^c, T^j)$ are independent}
		This case becomes the SV2 model in \cite{Huang:in} when $v^c_0=1$ and $\rho=a^c_2=\mu^c=\sigma^c=0$. Let $\rho=0$ and $a^c_2=0$; then,
		\begin{align}
		\bbE e^{-rY_t}&=\bbE\edg{\exp\brak{-r\brak{X^j_{T^j_t}+a^c_1T^c_t+a^c_3B^j_{T^c_t}}}\bbE\edg{\exp\brak{-ra^c_4 W^c_{T^c_t}}\vert X^j,B^c B^j}}\\
		&=\bbE\edg{\exp\brak{-r\brak{(a^c_1-r|a^c_4|^2/2)T^c_t+X^j_{T^j_t}+a^c_3B^j_{T^c_t}}}}\\
		&=\bbE\edg{\exp\brak{-r\brak{(a^c_1-r|a^c_4|^2/2)T^c_t+a^c_3B^j_{T^c_t}}}\scL_J\brak{ra^j_1-\frac{r^2|a^j_2|^2}{2}, T^j_t}}.
		\end{align}
		Note that
		\begin{align}
		\scF(t,s)&:=\bbE\edg{a^c_3B^j_{s}\scL_J\brak{ra^j_1-\frac{r^2|a^j_2|^2}{2}, T^j_t}}\\
		&=\iint a^c_3z\scL_J\brak{ra^j_1-\frac{r^2|a^j_2|^2}{2}, y}f_{T^j_t,B^j_s}(y,z)dydz,
		\end{align}
		where we can calculate $f_{T^j_t,B^j_s}$ by Proposition \ref{prep_joint1}. Then,
		\begin{align}
		\bbE e^{-rY_t}=\bbE\edg{\exp\brak{-r\brak{(a^c_1-r|a^c_4|^2/2)T^c_t}}\scF(t,T^c_t)}.
		\end{align}
		\subsubsection{When $T^c=T^j$}
		This is the SV3 case in \cite{Huang:in} when $\rho=1, a^c_2=0, \mu^c=\mu^j$, and $\sigma^c=\sigma^j$. Then, 
		\begin{align}
		\bbE e^{-rY_t}&=\bbE\edg{\exp\brak{-rX^c_{T^j_t}}\bbE\edg{\exp\brak{-rX^j_{T^j_t}}\vert T^j, X^c}}\\
		&=\bbE\edg{\exp\brak{-r\brak{a^c_1T^j_t+a^c_3B^j_{T^j_t}+a^c_4W^c_{T^j_t}}}\scL_J\brak{ra^j_1-\frac{r^2|a^j_2|^2}{2}, T^j_t}}\\
		&=\bbE\edg{\exp\brak{-r\brak{(a^c_1-r|a^c_4|^2/2)T^j_t+a^c_3B^c_{T^j_t}}}\scL_J\brak{ra^j_1-\frac{r^2|a^j_2|^2}{2}, T^j_t}}.
		\end{align}
		We can calculate the last expression using the following formula and our result from calculating $f_{T_t,B_j}$ using Proposition \ref{prep_joint1}.
		\begin{align}
		\bbE g(T^j_t)\exp\brak{-r B^j_{T^j_t}}&=\iint g(y)e^{-rz}f_{T^j_t,B^j_{T^j_t}}(y,z)dydz\\
		&=\iint g(y)e^{-rz}f_{T^j_t,B^j_{y}}(y,z)dydz.
		\end{align}
	}
	\section{A note on the implementation}
		In order to find the transitional density of a correlated TCL process, one can employ the following steps: (i) solve \eqref{FPq} or \eqref{FPqFourier}, and (ii) apply Theorem \ref{mainthm0}. On the other hand, if one wants to find the Laplace transform of the correlated TCL process, we propose the following steps: (i') solve \eqref{FKG} or \eqref{FTG} based on the model, (ii') find $\bbE\edg{\exp\brak{-i\xi T^{(j,r)}_t}}$ using Proposition \ref{propG} or \ref{propFTG}, and (iii') apply Theorem \ref{mainthm}. 
		Each method has their own strengths and weaknesses. The former approach requires more computational power to calculate the PDE, but the latter approach needs an inverse transform.

	\section{Summary and future research}
	\edit{In this paper, we discussed why the framework in \cite{Carr:2004hl} is extremely restrictive and their proposed specifications are not compatible with their assumptions. On the other hand, we derived an alternative framework in which we model the time changes by an adapted process without being a stopping time. We checked the validity of our model by providing a no arbitrage condition and indicating how one can hedge contingent claims in our framework. We also found the distribution of the correlated TCL processes and studied the relationship to \cite{Huang:in}. Given the generality of the specifications in \cite{Huang:in}, our approach is applicable to virtually all of the continuous-time models proposed in the option pricing literature. 
		
		
		
		There are two directions for future investigations. First, one can implement the theoretical results empirically and investigate the empirical performance with the correlated TCL processes in time series analysis and option pricing. In this case, one also needs to analyze the numerical aspects of the PDEs in this article. In other words, one needs to find an efficient numerical method to find the PDEs in this article and to analyze the convergence speed and estimate the error. Another direction of future research is to explore nontrivial models for the time changes that make them stopping times. For such time changes, \cite{Carr:2004hl} provide an analytic formula for the characteristic function of a correlated TCL process when the \levy and the time changes are not independent, and when the time changes are stopping times. 
	}

\newpage
\appendix

\begin{figure}[!ht]
	\renewcommand{\thefigure}{\arabic{figure}}
	\begin{center}
		\includegraphics[width=0.9\textwidth]{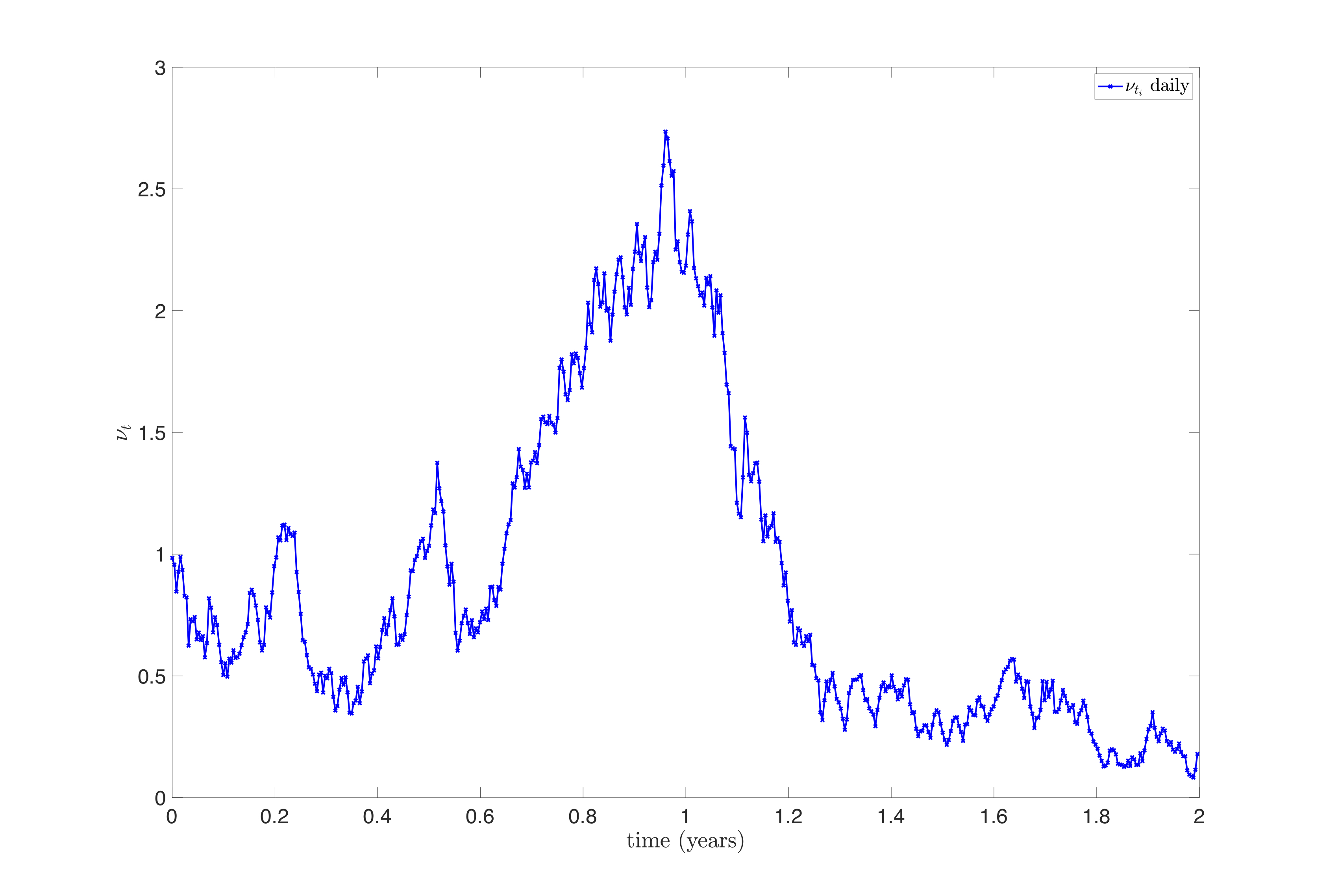}  
		\includegraphics[width=0.9\textwidth]{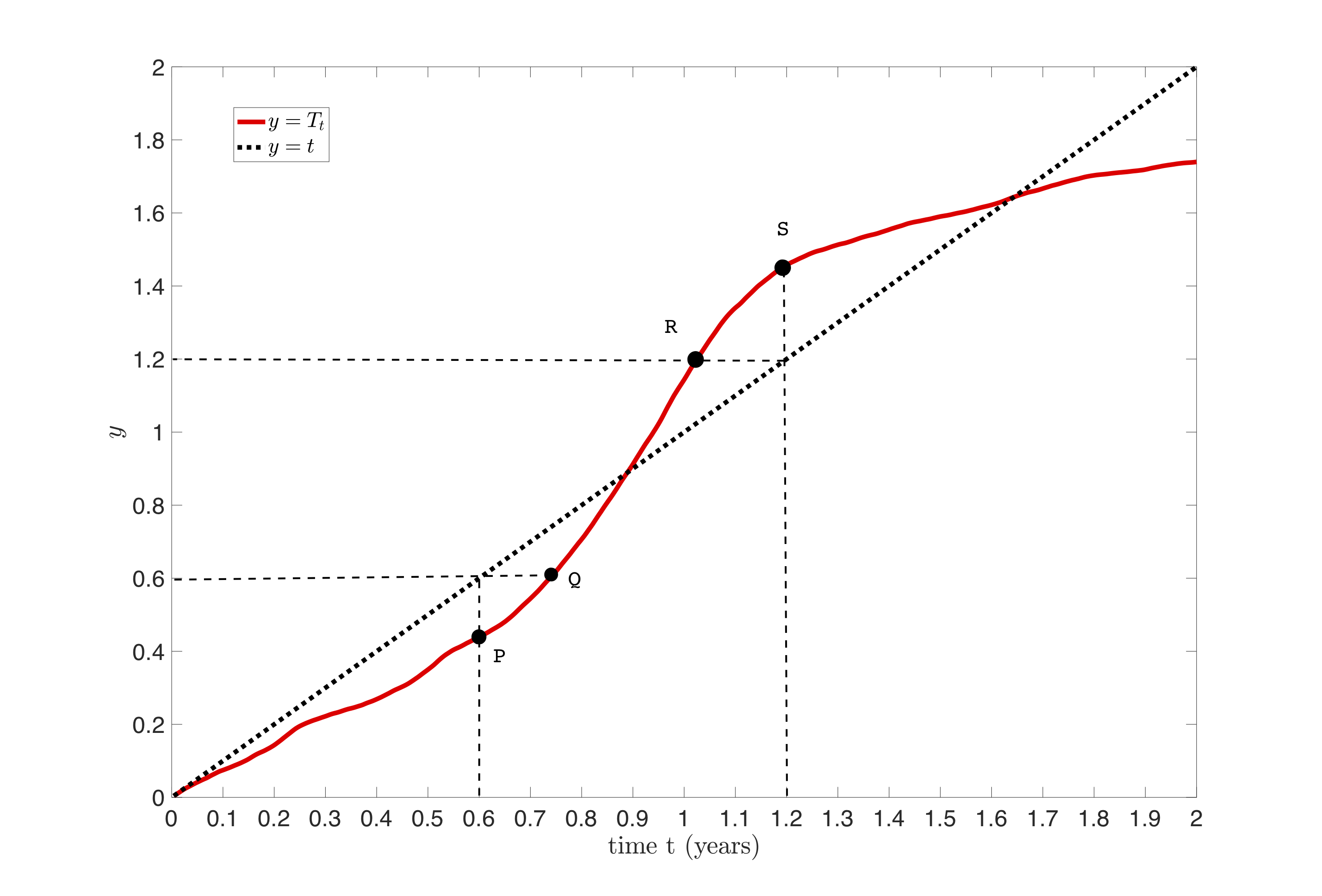}
	\end{center}     
	\caption{(top) Two years simulated Cox-Ingersoll-Ross (CIR) process, $ v = (v_t)_{t\geq0}$. (down) Integrated CIR process. $ T_t=\int_{0}^{t} v_s ds$.}\label{simu_cir}
\end{figure} 

\let\clearpage\relax
\setstretch{1} 
\section{Skimming of \cite{Carr:2004hl}}

In \cite{Carr:2004hl}, they assumed the followings in Section 2:
\begin{itemize}
	\setlength{\itemindent}{-.2in}
	\item The probability space $(\Omega,\scF,\bbP)$, endowed with a standard complete filtration $\bbF=(\scF_t)_{t\geq0}$.
	\item $X$ is a {\levy} process which is adapted to $\bbF$ with $\bbE e^{i\theta X_t}=e^{-t\Psi_X(\theta)}$.
	\item $T$ is a nondecreasing right-continuous and left-limit process such that, for any given $t$, $T_t$ is a stopping time with respect to $\bbF$, $T_t$ is finite almost surely, $\bbE[T_t]=t$. In addition, we assume $T_t\xrightarrow{t\to\infty}\infty$ almost surely.
\end{itemize}
Under these assumptions, \cite{Carr:2004hl} define $ Z_t(\theta):=\exp\brak{i\theta^T X_t+t\Psi_X(\theta)} $,
\begin{align}\label{tcl_process}
Y_t:= X_{T_t}
\end{align}
and they proposed the following lemma and theorem.
\begin{lemma}[Lemma 1 of \cite{Carr:2004hl}]
	For every $\theta\in\scD$, $M_t(\theta):=Z_{T_t}(\theta)$ is a complex-valued $\bbP$-martingale with respect to $(\scF^{Y,T}_t)_{t\geq0}$, the filtration generated by the process $\crl{(Y_t,T_t):t\geq 0}$.
\end{lemma}
\begin{theorem}[Theorem 1 of \cite{Carr:2004hl}]
	Let $\bbE$ and $\bbE^\theta$ be expectations under measure $\bbP$ and $\bbQ(\theta)$, respectively, where
	$\left.\frac{d\bbQ(\theta)}{d\bbP}\right|_t:=M_t(\theta).$
	Then we have the following formula:
	\begin{align}
	\bbE[e^{i\theta^\intercal Y_t}]=\bbE^\theta[e^{-T_t\Psi_X(\theta)}].
	\end{align}
\end{theorem}
Above theorem implies that one may calculate characteristic function of $Y$, regarding $T$ and $X$ are independent under the new complex measure. These results are the foundation of their paper and it has been widely used in vast literature such as \cite{Huang:in}.

Their proof is based on three steps: (i) prove that $Z(\theta)$ is a martingale under $\bbF$; (ii) apply the optional stopping theorem for stopping time $T_t$ to $Z(\theta)$ to conclude $M$ is a martingale with respect to $(\scF^{Y,T}_s)_{s\geq0}$, and; (iii) prove Theorem 1 using the fact that $M(\theta)$ is a density process.

\section{Proofs}
\textit{Proof of Proposition \ref{mainthm1}} Note that
\begin{align}
f_{W_{J_{T_t}},B_{J_{T_t}},J_{T_t}}(w,b,j)&=f_{W_j,B_j,J_{T_t}}(w,b,j)
=f_{W_j}(w)f_{{B_j},J_{T_t}}(b,j)
\end{align}
since $W_j$ and $(B_j,J_{T_t})$ are independent. Therefore,
\begin{align}
f_{Z_{J_{T_t}}}(z)&=\frac{1}{\rho\sqrt{1-\rho^2}}\iint f_{W_{J_{T_t}},B_{J_{T_t}},J_{T_t}}\brak{\frac{z-b}{\sqrt{1-\rho^2}},\frac{b}{\rho},j}djdb\\
&=\frac{1}{\rho\sqrt{1-\rho^2}}\iint f_{W_{j}}\brak{\frac{z-b}{\sqrt{1-\rho^2}}}f_{{B_j},J_{T_t}} \brak{\frac{b}{\rho},j}dbdj\\
&=\frac{1}{\sqrt{1-\rho^2}}\iint f_{W_{j}}\brak{\frac{z-\rho\tilde b}{\sqrt{1-\rho^2}}}f_{{B_j},J_{T_t}} \brak{\tilde b,j}d\tilde bdj
\end{align}
where $\tilde b=b/\rho$.
On the other hand, since $J$ is independent to $T$, we have
\begin{align*}
f_{B_j,J_{T_t}}(b,j)&=\int f_{B_j,J_y,T_t}(b,j,y)dy\\
&=\int f_{B_j|J_y,T_t}(b|j,y)f_{J_y|T_t}(j|y)f_{T_t}(y)dy\\
&=\int f_{B_{j}|J_y,T_t}(b|j,y)f_{J_y}(j)f_{T_t}(y)dy
\end{align*}
Since $\sigma(B_j,T_t)$ and $J_y$ are independent, we have $f_{B_{j}|J_y,T_t}=f_{B_{j}|T_t}$ and therefore,
\[
f_{B_{J_{T_t}},J_{T_t}}(b,j)
=\int f_{J_y}(j)f_{B_{j},T_t}(b,y)dy.
\]
In sum,
\begin{align}
f_{Z_{J_{T_t}}}(z)
&=\iiint \frac{1}{\sqrt{2\pi j(1-\rho^2)}}e^{-\frac{(z-\rho b)^2}{2(1-\rho^2)j}} f_{B_j,T_t}( b,y)f_{J_y}(j)dyd bdj.
\end{align}
\vspace*{1cm}\QEDB

\noindent\textit{Proof of Proposition \ref{prep_joint1}} Note that
\begin{align*}
q(t,x,y,z)&:=f_{v_t,T_t,B_t}(x,y,z); &t\leq j
\end{align*}
is a solution to \eqref{FPq}.
For $t\leq j$,
\begin{align*}
f_{T_t,B_j}(y,z)&=\iint f_{B_j|B_t,v_t,T_t}(z|\tilde z,x,y)f_{B_t,v_t,T_t}(\tilde z,x,y) dxd\tilde z\\
&=\iint f_{B_j-B_t}(z-\tilde z)f_{B_t,v_t,T_t}(\tilde z,x,y) dxd\tilde z\\
&=\frac{1}{\sqrt{2\pi(j-t)}}\iint e^{-\frac{(z-\tilde z)^2}{2(j-t)}}q(t,x,y,\tilde z) dxd\tilde z
\end{align*}
since $T_t,B_t\in\scF^B_t$ are independent to $B_j-B_t$. 
On the other hand, for $t\geq j$, we have
\begin{align}
\bbE\edg{e^{i\xi T_t}\vert B_j}&=\left.\int e^{i\xi y}f_{T_t|B_j}(y|z)dy\right|_{z=B_j}
\end{align}
and
\begin{align}
\bbE\edg{e^{i\xi T_t}\vert B_j}&=\bbE\edg{\left.\bbE\edg{e^{i\xi T_t}\vert \scF_j} \right\vert B_j}=\bbE\edg{\exp(\Phi(t,j,\xi,v_j,T_j))\vert B_j}\\
&=\left.\iint \exp(\Phi(t,j,\xi,x,y))f_{v_j,T_j|B_j}(x,y|z)dxdy\right|_{z=B_j}\\
&=\left.\iint \sqrt{2\pi j}\exp\brak{\Phi(t,j,\xi,x,y)+\frac{z^2}{2j}}q(j,x,y,z)dxdy\right|_{z=B_j}.
\end{align}
Therefore,
\begin{align}
f_{T_t,B_j}(y,z)&=f_{T_t|B_j}(y|z)f_{B_j}(z)\\
&=\frac{1}{2\pi}\int e^{-i\xi  y}\iint\exp\brak{\Phi(t,j,\xi,x,\tilde y)}q(j,x,\tilde y,z)dxd\tilde yd\xi\\
&=\frac{1}{2\pi}\iiint \exp\brak{\Phi(t,j,\xi,x,\tilde y)-i\xi y}q(j,x,\tilde y,z)dxd\tilde yd\xi.
\end{align}
\vspace*{1cm}\QEDB

\noindent\textit{Proof of Theorem \ref{mainthm0}} Note that, for $\hat f_{T_t,B_j}(\xi,\eta):=\iint e^{-i\xi y-i\eta z}f_{T_t,B_j}(y,z)  dydz$,
\begin{align*}
f_{B_j,T_t}(b,y)=\frac{1}{(2\pi)^2}\iint e^{i\xi y+i\eta b} \hat f_{T_t,B_j}(\xi,\eta)d\xi d\eta.
\end{align*}
Therefore,
\begin{align*}
f_{Z_{J_{T_t}}}(z)&=\iiint \frac{1}{\sqrt{2\pi j(1-\rho^2)}}e^{-\frac{(z-\rho b)^2}{2(1-\rho^2)j}}f_{B_j,T_t}(b,y)f_{J_y}(j)dyd bdj\\
&=\frac{1}{(2\pi)^2}\iiint \iint \frac{1}{\sqrt{2\pi j(1-\rho^2)}}e^{-\frac{(z-\rho b)^2}{2(1-\rho^2)j}}e^{i\xi y+i\eta b} \hat f_{T_t,B_j}(\xi,\eta)f_{J_y}(j)d\xi d\eta dyd bdj\\
&=\frac{1}{(2\pi)^2}\iiint \brak{\int e^{i\eta b}\frac{1}{\sqrt{2\pi j(1-\rho^2)}}e^{-\frac{(z-\rho b)^2}{2(1-\rho^2)j}} db}\hat f_{T_t,B_j}(\xi,\eta)\brak{\int e^{i\xi y}f_{J_y}(j)dy} d\xi d\eta dj\\
&=\frac{1}{(2\pi)^2\rho}\iiint \hat f_{T_t,B_j}(\xi,\eta)\exp\brak{\frac{i\eta z}{\rho}-\frac{(1-\rho^2)\eta^2j}{2\rho^2}}\brak{\int e^{i\xi y}f_{J_y}(j)dy} d\xi d\eta dj.
\end{align*}
On the other hand, if $t\leq j$,
\begin{align*}
\hat f_{T_t,B_j}(\xi,\eta)&=\frac{1}{\sqrt{2\pi(j-t)}}\iiiint e^{-i\xi y-i\eta z}e^{-\frac{(z-\tilde z)^2}{2(j-t)}}q(t,x,y,\tilde z) dxd\tilde z dy dz\\
&=e^{-\frac{(j-t)\eta^2}{2}}\iiint e^{-i\xi y-i\eta\tilde z}q(t,x,y,\tilde z) dxd\tilde z dy\\
&=e^{-\frac{(j-t)\eta^2}{2}}\int\hat q(t,x,\xi,\eta) dx.
\end{align*}
In the case where $t>j$, by noting the calculation of $\bbE\edg{e^{-i\xi T_t}\vert B_j}$ in the proof of Proposition \ref{prep_joint1},
\begin{align*}
\hat f_{T_t,B_j}(\xi,\eta)&=\bbE e^{-i\xi T_t-i\eta B_j}=\bbE \edg{e^{-i\eta B_j}\bbE\edg{e^{-i\xi T_t}\vert B_j}}\\
&=\int e^{-i\eta z}\brak{\iint \sqrt{2\pi j}\exp\brak{\Phi(t,j,-\xi,x,y)+\frac{z^2}{2j}}q(j,x,y,z)dxdy}f_{B_j}(z)dz\\
&=\iint \exp\brak{\Phi(t,j,-\xi,x,y)}\brak{\int e^{-i\eta z}q(j,x,y,z)dz}dxdy\\
&=\frac{1}{2\pi}\iint \exp\brak{\Phi(t,j,-\xi,x,y)}\int e^{i\bar \xi y}\hat q(j,x,\bar{\xi},\eta)d\bar\xi dxdy\\
&=\frac{1}{2\pi}\iint \brak{\int \exp\brak{\Phi(t,j,-\xi,x,y)+i\bar \xi y}dy}\hat q(j,x,\bar{\xi},\eta)d\bar\xi dx.
\end{align*}
\vspace*{1cm}\QEDB

\noindent\textit{Proof of Theorem \ref{main_general}}
Note that
\begin{align}
f_{Y_t}(\tilde y)&=\frac{1}{\beta\sqrt{1-\rho^2}}\iint f_{W_{J_{T_t}},B_{J_{T_t}}, J_{T_t}}\brak{\frac{1}{\beta\sqrt{1-\rho^2}}\brak{\tilde y-\rho\beta b-\alpha j},b,j}dbdj\\
&=\frac{1}{\beta\sqrt{1-\rho^2}}\iint f_{W_j}\brak{\frac{1}{\beta\sqrt{1-\rho^2}}\brak{\tilde y-\rho \beta b-\alpha j}}f_{B_j, J_{T_t}}\brak{b,j}dbdj\\
&=\frac{1}{\beta\sqrt{1-\rho^2}}\iiint\frac{1}{\sqrt{2\pi j}}\exp\brak{-\frac{(\tilde y-\rho\beta b-\alpha j)^2}{2\beta^2j(1-\rho^2)}}f_{B_j,T_t}(b,y)f_{J_y}(j)dydbdj
\end{align}
Then our claim follows by the same argument in the proof of Theorem \ref{mainthm0}.
\vspace*{1cm}\QEDB

\noindent\textit{Proof of Theorem \ref{mainthm}}
Let
\begin{align*}
q(t,x,y,z)&:=f_{v_t,T_t,B_t}(x,y,z); &j\geq t\\
p(j,t,x,y,b)&:=f_{v_t,T_t|B_j}(x,y|b); &j\leq t
\end{align*}
\begin{align*}
\bbE& e^{-rZ_{J_{T_t}}} = \bbE\edg{e^{-r\rho B_{J_{T_t}}}\bbE\edg{e^{-r\sqrt{1-\rho^2}W_{J_{T_t}}}\vert J,B}}\\
&=\iiint\exp\brak{-r\rho b+\frac{r^2(1-\rho^2)j}{2}}f_{B_{j},T_t}(b,y)f_{J_y}(j)dy dbdj\\
&=\int_0^t\frac{1}{\sqrt{2\pi j}}\iiint e^{-r\rho b+\frac{r^2(1-\rho^2)j}{2}-\frac{b^2}{2j}}p(j,t,x,y,b)f_{J_y}(j)dxdydbdj\\
&+\int_t^\infty\frac{1}{\sqrt{2\pi(j-t)}}\iiiint e^{-r\rho b+\frac{r^2(1-\rho^2)j}{2}-\frac{(b-z)^2}{2(j-t)}}q(t,x,y,z) f_{J_y}(j)dxdydzdbdj
\end{align*}
Note that
\begin{align*}
\frac{1}{\sqrt{2\pi(j-t)}}\int e^{-r\rho b-\frac{(b-z)^2}{2(j-t)}}db=\exp\brak{-r\rho z+\frac{r^2\rho^2(j-t)}{2}}
\end{align*}
and therefore, when $j>t$, by integrating with respect to $b$,
\begin{align*}
&\frac{1}{\sqrt{2\pi(j-t)}}\iiiint e^{-r\rho b+\frac{r^2(1-\rho^2)j}{2}-\frac{(b-z)^2}{2(j-t)}}q(t,x,y,z) f_{J_y}(j)dxdydzdb\\
&=\iiint \exp\brak{\frac{r^2(1-\rho^2)j}{2}-r\rho z+\frac{r^2\rho^2(j-t)}{2}}q(t,x,y,z) f_{J_y}(j)dxdydz\\
&=\iiint \exp\brak{-r\rho z+\frac{r^2(j-\rho^2 t)}{2}}q(t,x,y,z) f_{J_y}(j)dxdydz
\end{align*}
By changing the integration parameter $b$ to $z$ for the first term, we have,
\begin{align*}
\bbE e^{-rZ_{J_{T_t}}}&=\int_0^t\frac{1}{\sqrt{2\pi j}}\iiint \exp\brak{-r\rho z+\frac{r^2(1-\rho^2)j}{2}-\frac{z^2}{2j}}p(j,t,x,y,z)f_{J_y}(j)dxdydzdj\\
&+\int_t^\infty\iiint \exp\brak{-r\rho z+\frac{r^2(j-\rho^2 t)}{2}}q(t,x,y,z) f_{J_y}(j)dxdydz dj\\
&=\iiiint \exp\brak{-r\rho z+\frac{r^2(j-(j\wedge t)\rho^2)}{2}}Q(j,t,x,y,z)f_{J_y}(j)dxdydzdj
\end{align*}
where
\begin{align*}
Q(j,t,x,y,z):&=\frac{1}{\sqrt{2\pi j}}e^{-\frac{z^2}{2j}}p(j,t,x,y,z)\indicator{j\leq t}+q(t,x,y,z)\indicator{j>t}\\
&=f_{v_t,T_t,B_{j\wedge t}}(x,y,z).
\end{align*}
If we simplify above equation by considering $g(y,j):=f_{J_y}(j)$ as a distribution with variables $(y,j)$,
\begin{align*}
\bbE e^{-rZ_{N_{Y_t}}}&=\int e^{r^2j/2}e^{-r^2\rho^2(j\wedge t)/2} \bbE\edg{g(T_t,j)e^{-r\rho B_{j\wedge t}}} dj.
\end{align*}
Let us denote $\bbE^{(j,r)}$ be the expectation under changed measure
\begin{align*}
d\bbP^{(j,r)}
=\exp\brak{-r\rho B_{j\wedge t}-\half r^2\rho^2(j\wedge t)}d\bbP=:\scE^{(j,r)}_td\bbP
\end{align*}
and $f^{(j,r)}_\xi$ to be the probability density of a random variable $\xi$ under $\bbP^{(j,r)}$.
Under $\bbP^{(j,r)}$, since $\scE^{(j,r)}_t$ is uniform integrable martingale by Novikov condition, $B^{(j,r)}:= B+r\rho1_{[0,j]}(\cdot)$ is a Brownian motion on $[0,t]$ such that 
\begin{align*}
\sigma\brak{B^{(j,r)}_s:s\leq u}=\sigma\brak{B_s:s\leq u}
\end{align*}
for all $u\leq t$.
Note that, for any $\tilde j$,
\begin{align*}
\bbE\edg{g(T_t,\tilde j)\scE^{(j,r)}_t}=\bbE^{(j,r)}\edg{g(T_t,\tilde j)}=\int g(y,\tilde j)f^{(j,r)}_{T_t}(y)dy=\int f_{J_y}(\tilde j)f^{(j,r)}_{T_t}(y)dy.
\end{align*}
Therefore, 
\begin{align*}
\bbE e^{-rZ_{J_{T_t}}}&=\int e^{\frac{r^2j}{2}}\bbE^{(j,r)}\edg{g(T_t,j)}dj=\iint e^{\frac{r^2j}{2}}f_{J_y}( j)f^{(j,r)}_{T_t}(y)dydj
\end{align*}

Note that
\begin{align}
f^{(j,r)}_{T_t}(y)=\frac{1}{2\pi}\int e^{i\xi y} \hat f^{(j,r)}_{T_t}(\xi)d\xi
\end{align}
where
\begin{align}
\hat f^{(j,r)}_{T_t}(\xi):=\int e^{-i\xi \tilde y}f^{(j,r)}_{T_t}(\tilde y)d\tilde y=\bbE^{(j,r)}\edg{\exp\brak{-i\xi T_t}}.
\end{align}
Note that, for $\bbP^{(j,r)}$-Brownian motion $B^{(j,r)}:=B+\int_0^\cdot r\rho{\rm\bf 1}_{[0,j]}(s)ds$,
\begin{align*}
dv_u&=\brak{\mu(u,v_u)-r\rho \sigma(u,v_u){\rm\bf 1}_{[0,j]}(u)}du+\sigma(u,v_u)dB^{(j,r)}_u\\
dT_u&=v_udu
\end{align*}

Then, the distribution of $(v^{(j,r)}_t,T^{(j,r)}_t)$ under $\bbP$ and the distribution of $(v_t,T_t)$ under $\bbP^{(j,r)}$ are identical. Therefore,
\begin{align}
\bbE^{(j,r)}\edg{\exp\brak{-i\xi T_t}}&=\bbE\edg{\exp\brak{-i\xi T^{(j,r)}_t}}.
\end{align}
In sum, we have
\begin{align*}
\bbE e^{-rZ_{J_{T_t}}}&=\iint e^{\frac{r^2j}{2}}f_{J_y}(j)\frac{1}{2\pi}\int e^{i\xi y}\bbE\edg{\exp(-i\xi T^{(j,r)}_t)}d\xi dydj\\
&=\frac{1}{2\pi}\iint e^{\frac{r^2j}{2}}\brak{\int f_{J_y}(j) e^{i\xi y} dy}\bbE\edg{\exp(-i\xi T^{(j,r)}_t)}d\xi dj
\end{align*}
\vspace*{1cm}\QEDB

\noindent\textit{Proof of Proposition \ref{propG}} 
The second part of the claim is obvious since $G(t,x,y,z)$ is the PDF of $(v^{(j,r)}_t, T^{(j,r)}_t)$ when $t\leq j$.
Consider the case when $t>j$. Since SDE \eqref{sdechanged} is identical to \eqref{sde} for $u>j$, we have
\begin{align}
\bbE\edg{\left.\exp\brak{-i\xi T^{(j,r)}_t}\right|\scF_j}=\exp(\Phi(t,j,-\xi,v^{(j,r)}_j,T^{(j,r)}_j))
\end{align}
by our definition of $\Phi$. If we take the expectation on both side, we have the first part of the claim.
\vspace*{1cm}\QEDB

\noindent\textit{Proof of Preposition \ref{propFTG}} Note that
\begin{align}
\bbE\exp(\Phi(t,j,-\xi,v^{(j,r)}_j,T^{(j,r)}_j))&=\bbE\exp\brak{\phi(t,j,-\xi,v^{(j,r)}_j)+i\psi(t,j,-\xi,v^{(j,r)}_j)T^{(j,r)}_j}\\
&=\iint \exp\brak{\phi(t,j,-\xi,x)+i\psi(t,j,-\xi,x)y}G(j,x,y)dydx\\
&=\int e^{\phi(t,j,-\xi,x)}\int\exp\brak{i\psi(t,j,-\xi,x)y}G(j,x,y)dydx
\end{align}
Note that
\begin{align}
\hat G(u,x,\eta):=\int e^{-i\eta y}G(u,x,y)dy
\end{align}
satisfies \eqref{FTG} and 
\begin{align}
\int\exp\brak{i\psi(t,j,-\xi,x)y}G(j,x,y)dy&=\hat G(j,x,-\psi(t,j,-\xi,x))\\
\int \exp\brak{-i\xi y}G(t,x,y)dy&=\hat G(t,x,\xi).
\end{align}
\vspace*{1cm}\QEDB

\noindent\textit{Proof of Theorem \ref{mgf_generalmodel}}
We can follow the steps from previous results. For $Q$ defined as in the proof of Theorem \ref{mainthm},
\begin{align}
\bbE e^{-rY_t}&=\bbE\edg{\bbE\edg{e^{-r\alpha J_{T_t}-r\beta Z_{J_{T_t}}}\vert J,B}}\\
&=\bbE\edg{e^{-r\alpha J_{T_t}-r\beta\rho B_{J_{T_t}}}\bbE\edg{e^{-r\beta \sqrt{1-\rho^2}W_{J_{T_t}}}\vert J,B}}\\
&=\bbE\edg{e^{-r\alpha J_{T_t}-r\beta\rho B_{J_{T_t}}+\frac{r^2\beta^2(1-\rho^2)}{2}J_{T_t}}}\\
&=\bbE\edg{\exp\brak{-r\beta\rho B_{J_{T_t}}+\brak{\frac{r^2\beta^2(1-\rho^2)}{2}-r\alpha}J_{T_t}}}\\
&=\iiint\exp\brak{-r\alpha j-r\beta\rho b+\frac{r^2\beta^2(1-\rho^2)}{2}j}f_{B_{j},T_t}(b,y)f_{J_y}(j)dy dbdj\\
&=\iiiint \exp\brak{-r\alpha j-r\beta\rho z+\frac{r^2\beta^2(j-(j\wedge t)\rho^2)}{2}}Q(j,t,x,y,z)f_{J_y}(j)dxdydzdj\\
&=\int e^{r^2\beta^2j/2-r\alpha j}e^{-r^2\beta^2\rho^2(j\wedge t)/2} \bbE\edg{g(T_t,j)e^{-r\beta\rho B_{j\wedge t}}} dj\\
&=\frac{1}{2\pi}\iint e^{r^2\beta^2j/2-r\alpha j}\brak{\int e^{i\xi y}f_{J_y}( j)dy}\bbE\edg{\exp\brak{-i\xi T^{(j,r\beta)}_t}}d\xi dj.
\end{align}
Here, $g(y,j):=f_{J_y}(j)$.
\vspace*{1cm}\QEDB


\section{Proof of Remark \ref{mainrmk}}
The argument is analogous to the proof of Theorem \ref{mainthm}. From our calculation in the proof of Theorem \ref{mainthm}, we have
\begin{align*}
\bbE& e^{i\theta Z_{J_{T_t}}}
=\iiint\exp\brak{i\theta\rho b-\frac{\theta^2(1-\rho^2)j}{2}}f_{B_{j},T_t}(b,y)f_{J_y}(j)dy dbdj\\
&=\int_0^t\frac{1}{\sqrt{2\pi j}}\iiint e^{i\theta\rho b-\frac{\theta^2(1-\rho^2)j}{2}-\frac{b^2}{2j}}p(j,t,x,y,b)f_{J_y}(j)dxdydbdj\\
&+\int_t^\infty\frac{1}{\sqrt{2\pi(j-t)}}\iiiint e^{i\theta\rho b-\frac{\theta^2(1-\rho^2)j}{2}-\frac{(b-z)^2}{2(j-t)}}q(t,x,y,z) f_{J_y}(j)dxdydzdbdj
\end{align*}
Note that
\begin{align*}
\frac{1}{\sqrt{2\pi(j-t)}}\int e^{i\theta\rho b-\frac{(b-z)^2}{2(j-t)}}db=\exp\brak{i\theta\rho z-\frac{\theta^2\rho^2(j-t)}{2}}
\end{align*}
and therefore, when $j>t$, by integrating with respect to $b$,
\begin{align*}
&\frac{1}{\sqrt{2\pi(j-t)}}\iiiint e^{i\theta\rho b-\frac{\theta^2(1-\rho^2)j}{2}-\frac{(b-z)^2}{2(j-t)}}q(t,x,y,z) f_{J_y}(j)dxdydzdb\\
&=\iiint \exp\brak{-\frac{\theta^2(1-\rho^2)j}{2}+i\theta\rho z-\frac{\theta^2\rho^2(j-t)}{2}}q(t,x,y,z) f_{J_y}(j)dxdydz\\
&=\iiint \exp\brak{i\theta\rho z-\frac{\theta^2(j-\rho^2 t)}{2}}q(t,x,y,z) f_{J_y}(j)dxdydz
\end{align*}
By changing the integration parameter $b$ to $z$ for the first term, we have,
\begin{align*}
\bbE e^{i\theta Z_{J_{T_t}}}&=\int_0^t\frac{1}{\sqrt{2\pi j}}\iiint \exp\brak{i\theta\rho z-\frac{\theta^2(1-\rho^2)j}{2}-\frac{z^2}{2j}}p(j,t,x,y,z)f_{J_y}(j)dxdydzdj\\
&+\int_t^\infty\iiint \exp\brak{i\theta\rho z-\frac{\theta^2(j-\rho^2 t)}{2}}q(t,x,y,z) f_{J_y}(j)dxdydz dj\\
&=\iiiint \exp\brak{i\theta\rho z-\frac{\theta^2(j-(j\wedge t)\rho^2)}{2}}Q(j,t,x,y,z)f_{J_y}(j)dxdydzdj
\end{align*}
where
\begin{align*}
Q(j,t,x,y,z):&=\frac{1}{\sqrt{2\pi j}}e^{-\frac{z^2}{2j}}p(j,t,x,y,z){\rm\bf 1}_{j\leq t}+q(t,x,y,z){\rm\bf 1}_{j>t}\\
&=f_{v_t,T_t,B_{j\wedge t}}(x,y,z).
\end{align*}
If we simplify above equation by considering $g(y,j):=f_{J_y}(j)$ as a distribution with variables $(y,j)$,
\begin{align*}
\bbE e^{i\theta Z_{N_{Y_t}}}&=\int e^{-\theta^2j/2}e^{\theta^2\rho^2(j\wedge t)/2} \bbE\edg{g(T_t,j)e^{i\theta\rho B_{j\wedge t}}} dj\\
&=\int e^{-\theta^2j/2}\bbE\edg{g(T_t,j)\scE^{\theta,j}_t} dj.
\end{align*}
Then, we have the following proposition by the same technique in the proof of Theorem \ref{mainthm}.
\begin{proposition}
	\begin{align}
	\bbE\edg{g(T_t,j)\scE^{\theta,j}_t}=\frac{1}{2\pi}\int \brak{\int g(y,j)e^{i\xi y}dy}\bbE\edg{e^{-i\xi T_t}\scE^{\theta,j}_t}d\xi
	\end{align}
\end{proposition}
\begin{proof}
	Note that
	\begin{align*}
	&\bbE\edg{g(T_t,j)\scE^{\theta,j}_t}=\iint g(y,j)e^{i\theta\rho b+\frac{\theta^2\rho^2(j\wedge t)}{2}}f_{T_t,B_{j\wedge t}}(y,b)dydb\\
	&=\iint g(y,j)e^{i\theta\rho b+\frac{\theta^2\rho^2(j\wedge t)}{2}}\brak{\frac{1}{2\pi} \int e^{i\xi y}\int e^{-i\xi\tilde y}f_{T_t,B_{j\wedge t}}(\tilde y,b)d\tilde y d\xi}dydb\\
	&=\frac{1}{2\pi}\iiiint g(y,j)e^{i\xi y} e^{i\theta\rho b+\frac{\theta^2\rho^2(j\wedge t)}{2}}e^{-i\xi\tilde y}f_{T_t,B_{j\wedge t}}(\tilde y,b)d\tilde y  dydb d\xi\\
	&= \frac{1}{2\pi}\int \brak{\int g(y,j)e^{i\xi y}dy}\bbE\edg{e^{-i\xi T_t}\scE^{\theta,j}_t}d\xi
	\end{align*}
\end{proof}
\noindent Therefore, 
\begin{align}
\bbE e^{i\theta Z_{N_{Y_t}}}=\frac{1}{2\pi}\iint e^{-\theta^2j/2}\brak{\int f_{J_y}( j)e^{i\xi y}dy}\bbE\edg{e^{-i\xi T_t}\scE^{\theta,j}_t}d\xi dj
\end{align}

\section{Calculating $\bbE\edg{e^{-i\xi T_t}\scE^{\theta,j}_t}$ using Fokker Planck equation}
\begin{lemma} For $t>j$,
	\begin{align}
	\bbE\edg{e^{-i\xi T_t}\scE^{\theta,j}_t}=
	\bbE\edg{\exp(\Phi(t,j,\xi,v_j,T_j))\scE^{\theta,j}_j}
	\end{align}
\end{lemma}
\begin{proof}
	For $t>j$, we have $\scE^{\theta,j}_t=\scE^{\theta,j}_j\in\scF_j$ and therefore,
	\begin{align}
	\bbE\edg{\left.\exp\brak{-i\xi T_t}\scE^{\theta,j}_t\right|\scF_j}&=\scE^{\theta,j}_j\bbE\edg{\left.\exp\brak{-i\xi T_t}\right|\scF_j}\\
	&=\scE^{\theta,j}_j\exp(\Phi(t,j,\xi,v_j,T_j))
	\end{align}
	by our definition of $\Phi$. We prove the lemma by taking the expectation on both side.
\end{proof}
Above lemma tell us that we only need to know distribution of $(v_t,T_t,B_t)$ for $t\leq j$ to calculate
$\bbE\edg{e^{-i\xi T_t}\scE^{\theta,j}_t}$. 
\begin{proposition}
	Let \begin{align}
	G^\theta(t,x,y):=\int\exp\brak{i\theta\rho z+\frac{\theta^2\rho^2t}{2}}f_{v_t,T_t,B_t}(x,y,z)dz.
	\end{align}
	Then, $G^\theta$ satisfies
	\begin{align}\label{FKGappend}
	(\partial_u G^\theta)(u,x,y)&=\half\partial_{xx}\edg{|\sigma(u,x)|^2G^\theta(u,x,y)}-\partial_{x}\edg{\brak{\mu(u,x)+i\theta\rho\sigma(u,x)}G^\theta(u,x,y)}\\\notag
	&\quad-x\partial_yG^\theta(u,x,y)\\
	G^\theta(0,x,y)
	&=\delta_{v_0}(x)\delta_{T_0}(y)
	\end{align}
\end{proposition}
\begin{proof}
	Note that $H(t,x,y,z):=f_{v_t,T_t,B_t}(x,y,z)$ satisfies the Fokker Planck PDE:
	\begin{align}
	\partial_uH=\half \partial_{xx}\edg{\sigma^2H}+\half\partial_{zz} H+\partial_{xz}\edg{\sigma H}-\partial_x[\mu H]-\partial_y[xH].
	\end{align}
	If we multiply both side with $\exp\brak{i\theta\rho z+\frac{\theta^2\rho^2u}{2}}$ and integrate with respect to $z$, then the left hand side becomes
	\begin{align}
	\int\exp\brak{i\theta\rho z+\frac{\theta^2\rho^2u}{2}}(\partial_uH)(u,x,y,z)dz=	\partial_uG^\theta(u,x,y)-\frac{\theta^2\rho^2}{2}G^\theta(u,x,y).
	\end{align}
	On the other hand, we have
	\begin{align}
	\int\exp\brak{i\theta\rho z+\frac{\theta^2\rho^2u}{2}}(\partial_{z}H)(u,x,y,z)dz&=-i\theta\rho G^\theta(u,x,y)\\
	\int\exp\brak{i\theta\rho z+\frac{\theta^2\rho^2u}{2}}(\partial_{zz}H)(u,x,y,z)dz&=-\theta^2\rho^2 G^\theta(u,x,y)
	\end{align}
	by integration by parts. Therefore, the right-hand side becomes
	\begin{align}
	\half \partial_{xx}\edg{\sigma^2G^\theta}-\frac{\theta^2\rho^2}{2}G^\theta-i\theta\rho\partial_{x}\edg{\sigma G^\theta}-\partial_x[\mu G^\theta]-x\partial_y[G^\theta].
	\end{align}
	We proved the claim.
\end{proof}

\begin{proposition}\label{propG0} For solution $G^\theta$ of \eqref{FKGappend},
	\begin{align}
	\bbE\edg{e^{-i\xi T_t}\scE^{\theta,j}_t}=\begin{cases}
	\iint\exp\brak{\Phi(t,j,\xi,x,y)}G^\theta(j,x,y)dxdy&\text {if $t>j$}\\
	\iint \exp\brak{-i\xi y}G^\theta(t,x,y)dxdy&\text {if $t\leq j$}
	\end{cases}
	\end{align}
\end{proposition}
If $\Phi(t,s,\xi,x,y)$ is affine with respect to $y$, then the Fokker-Planck PDE can be simplified.
\begin{proposition}\label{propFTG0}
	Assume that
	\begin{align}
	\Phi(t,s,\xi,x,y)=\phi(t,s,\xi,x)+i\psi(t,s,\xi,x)y
	\end{align}
	In this case,
	\begin{align}
	\bbE\edg{\exp\brak{-i\xi T_t}\scE^{\theta,j}_t}=\begin{cases}
	\int e^{\phi(t,j,\xi,x)} \hat G^\theta(j,x,-\psi(t,j,\xi,x))dx &\text{ for $t>j$}\\
	\int \hat G^\theta(t,x,\xi) dx&\text{ for $t\leq j$}
	\end{cases}
	\end{align}
	where $\hat G^\theta$ satisfies $\hat G^\theta(0,x,\eta)=\delta_{v_0}(x)$ and 
	\begin{align}\label{FTGappend}
	(\partial_u \hat G^\theta)(u,x,\eta)&=\half\partial_{xx}\edg{|\sigma(u,x)|^2\hat G^\theta(u,x,\eta)}-\partial_{x}\edg{\brak{\mu(u,x)+i\theta \rho\sigma(u,x)}\hat G^\theta(u,x,\eta)}\\
	&\quad-ix\eta\hat G^\theta(u,x,\eta)
	\end{align}
\end{proposition}
\begin{proof}
	Let us define
	\begin{align}
	\hat G^\theta(t,x,\eta):=\int e^{-i\eta y}G(t,x,y,z)dy.
	\end{align}
	Then, it satisfies the Fourier transform of \eqref{FKGappend}, that is,
	\begin{align}
	(\partial_u \hat G^\theta)(u,x,\eta)&=\half\partial_{xx}\edg{|\sigma(u,x)|^2\hat G^\theta(u,x,\eta)}\\
	&\quad-\partial_{x}\edg{\brak{\mu(u,x)+i\theta \rho\sigma(u,x)}\hat G^\theta(u,x,\eta)}\\
	&\quad-ix\eta\hat G^\theta(u,x,\eta).
	\end{align}
	Note that
	\begin{align}
	&\iint\exp\brak{\Phi(t,j,\xi,x,y)}G^\theta(j,x,y)dxdy\\
	&=\int e^{\phi(t,s,\xi,x)}\int e^{i\psi(t,s,\xi,x)y}G^\theta(j,x,y)dydx
	\end{align}
	By the definition of $\hat G^\theta$,
	\begin{align}
	\int e^{i\psi(t,j,\xi,x)y}G^\theta(j,x,y)dy&=\hat G^\theta(j,x,-\psi(t,j,\xi,x))
	\end{align}
	and the claim is proved.
\end{proof}

\newpage\clearpage
\renewcommand{\baselinestretch}{1} 
\normalsize
\bibliographystyle{rfs}
\bibliography{untitled}

\end{document}